%% file: main.tex
\documentclass[12pt,reqno]{amsart}

\usepackage[a4paper, margin=1in]{geometry}

\usepackage{hyperref}
\hypersetup{
	colorlinks=true,
	linkcolor=blue,
	pdfpagemode=FullScreen,
}

\input{commands}

\title[Quasi-invariance Gaussian measures for NLS]{Quantitative quasi-invariance of Gaussian measures below the energy level for the 1D generalized nonlinear Schrödinger equation and application to global well-posedness}

\author{Alexis Knezevitch}

\begin{document}
	
	\address{ENS de Lyon site Monod
		UMPA UMR 5669 CNRS
		46, allée d’Italie
		69364 Lyon Cedex 07, FRANCE}
	
	\email{alexis.knezevitch@ens-lyon.fr}

	\begin{abstract}
	We consider the Schrödinger equation on the one dimensional torus with a general odd-power nonlinearity $p \geq 5$, which is known to be globally well-posed in the Sobolev space $H^\sigma(\mathbb{T})$, for every $\sigma \geq 1$, thanks to the conservation and finiteness of the energy. For regularities $\sigma < 1$, where this energy is infinite, we explore a globalization argument adapted to random initial data distributed according to the Gaussian measures $\mu_s$, with covariance operator $(1-\Delta)^s$, for $s$ in a range $(s_p,\frac{3}{2}]$. We combine a deterministic local Cauchy theory with the quasi-invariance of Gaussian measures $\mu_s$, with additional $L^q$-bounds on the Radon-Nikodym derivatives, to prove that the Gaussian initial data generate almost surely global solutions. These $L^q$-bounds are obtained with respect to Gaussian measures accompanied by a cutoff on a renormalization of the energy; the main tools to prove them are the Boué-Dupuis variational formula and a Poincaré-Dulac normal form reduction. This approach is similar in spirit to Bourgain's invariant argument~\cite{bourgain1994periodic} and to a recent work by Forlano-Tolomeo in~\cite{forlano2022quasi}.
\end{abstract}
	
	\maketitle

	\tableofcontents

	\section{Introduction}\input{introduction}

	\section{Preliminaries}\input{preliminaries}

	\section{Quantitative quasi-invariance and globalization of the local solutions}\label{section Quantitative quasi-invariance and globalization of the local solutions}\input{Quantitative_quasiinv_and_GWP}
	
	\section{Quantitative quasi-invariance for the original flow  and application}\label{section Quantitative quasi-invariance for the original flow  and application} \input{Quasiinv_for_the_original_flow}

    \section{Renormalized energy and associated cutoff Gaussian measures}\label{section Renormalized energy and associated cutoff Gaussian measures}\input{renormalized_energy_and_cutoff_gm} 

	\section{Poincaré-Dulac normal form reduction and deterministic estimate}\label{section Poincaré-Dulac normal form reduction and deterministic estimate}\input{Poincare_Dulac}

	\section{Uniform Lq integrability of the Radon-Nikodym derivatives}\label{section Uniform Lq integrability of the Radon-Nikodym derivatives}\input{uniform_Lp_bound_RND}

    \bibliographystyle{siam}
    \bibliography{refs_qualitative_quasiinv_and_gwp}

\end{document}

%% file: commands.tex
\usepackage{amssymb,amsfonts}
\usepackage[all,arc]{xy}
\usepackage{enumerate}
\usepackage{mathrsfs,stmaryrd}
\usepackage{relsize} 
\usepackage{enumitem}

\usepackage{pgf, tikz}

\usepackage{xcolor}

\newtheorem{thm}{Theorem}[section]
\newtheorem{cor}[thm]{Corollary}
\newtheorem{prop}[thm]{Proposition}
\newtheorem{lem}[thm]{Lemma}

\theoremstyle{definition}
\newtheorem{defn}[thm]{Definition}

\newtheorem{notn}[thm]{Notation}
\newtheorem{notns}[thm]{Notations}

\theoremstyle{definition}
\newtheorem{rem}[thm]{Remark}

\numberwithin{equation}{section}

\newcommand{\hsp}{\hspace{0.1cm}}

\newcommand{\cjg}[1]{%
  \overline{#1}%
  }                     

\renewcommand{\hat}{\widehat}

\newcommand{\lra}{\longrightarrow}
\newcommand{\ra}{\rightarrow}

\newcommand{\tendsto}[1]{%
\underset{#1}{\lra}%
}

\renewcommand{\Re}{\textnormal{Re}}
\renewcommand{\Im}{\textnormal{Im}}

\newcommand{\C}{\mathbb{C}}		
\newcommand{\E}{\mathbb{E}}		

\newcommand{\N}{\mathbb{N}}		

\renewcommand{\P}{\mathbb{P}}	
\newcommand{\R}{\mathbb{R}}		
\newcommand{\T}{\mathbb{T}}

\newcommand{\Z}{\mathbb{Z}}		

\usepackage{dsfont}
\renewcommand{\1}{\mathds{1}}			

\newcommand{\cB}{\mathcal B}		
\newcommand{\cC}{\mathcal C}	

\newcommand{\cE}{\mathcal E}	
\newcommand{\cF}{\mathcal F}		
\newcommand{\cG}{\mathcal G}	
\newcommand{\cM}{\mathcal M}	
\newcommand{\cN}{\mathcal N}	

\newcommand{\cT}{\mathcal T}		

\newcommand{\eps}{\varepsilon}

\newcommand{\norm}[1]{%
	\lnorm #1 \rnorm%
}
\newcommand{\triplenorm}[1]{{\left\vert\kern-0.25ex\left\vert\kern-0.25ex\left\vert #1 
		\right\vert\kern-0.25ex\right\vert\kern-0.25ex\right\vert}}

\newcommand{\p}{\partial}

\newcommand{\lnorm}{\left\lVert} 
\newcommand{\rnorm}{\right\rVert} 

\newcommand{\hsig}{H^{\sigma}}
\newcommand{\hsigt}{\hsig(\T)}

\newcommand{\lincp}{k_1-k_2+...-k_{p+1}=0}
\newcommand{\Omgz}{\Omega(\vec{k}) = 0}
\newcommand{\Omgk}{\Omega(\vec{k})}

\newcommand{\Omgnz}{\Omega(\vec{k}) \neq 0}

%% file: introduction.tex
In this paper, we consider the generalized nonlinear Schrödinger equation posed on the one dimensional torus $\T := \R / 2\pi \Z$,
\begin{equation}\label{pNLS}
	\begin{cases}
		i\p_t u + \p_x^2 u = |u|^{p-1}u, \hspace{.5cm} (t,x) \in \R \times \T \\
		u|_{t=0} = u_0 
	\end{cases}
	\tag{pNLS}
\end{equation}
for every $p \geq 5$ odd, and with initial data taken in the Sobolev spaces ($\sigma \in \R$):
\begin{equation*}
	\hsigt := \{u_0 : \T \to \C \hsp : \hsp \sum_{n \in \Z} \langle n \rangle^{2\sigma} |\hat{u_0}(n)|^2 < +\infty  \},
\end{equation*}
where $\langle n \rangle := (1+|n|^2)^{\frac{1}{2}}$ and $\hat{u_0}$ refers to the Fourier transform. More precisely, we will be interested by regularities below the "energy level", which means that $\sigma < 1$. Let us start by reviewing known results concerning the well-posedness of~\eqref{pNLS}. \\

\noindent \textbf{Local well-posedness :} To obtain local solutions to~\eqref{pNLS}, fixed point methods have been applied. It consists in working with the equivalent integral formulation of the equation, introducing for any initial data $u_0 \in \hsigt$ the \textit{Duhamel map} $\Gamma_{u_0}$, defined by:
\begin{equation}\label{Duhamel map}
	\Gamma_{u_0} [u] (t) = e^{it\p_x^2} u_0 - i \int_0^t e^{i(t-t') \p_x^2} \big[u(t')|u(t')|^{p-1}\big] dt'
\end{equation}
(where $e^{it\p_x^2}$ is the propagator of the linear Schrödinger equation) for which a fixed point is a solution, local in general. Following this approach, Bourgain proved in~\cite{Bourgain1993} that~\eqref{pNLS} is locally well-posed for all $\sigma > \frac{1}{2} - \frac{2}{p-1}$. In this work, dispersive technologies are used. The fixed point argument is performed with the so-called \textit{Bourgain spaces} $X^{\sigma,b}_\delta$ (for small time $\delta > 0$ and $b > \frac{1}{2}$), which are in particular spaces that embed into the space of continuous functions $\cC([-\delta,\delta],\hsig)$, and suitable inequalities are established through \textit{Strichartz estimates}. \\
However, since the full regime $\sigma > \frac{1}{2} - \frac{2}{p-1}$ is out of the reach of this paper, it is sufficient here to consider the regime $\sigma > \frac{1}{2}$, where $\hsigt$ is known to be an algebra, thereby making a more direct local theory available using only the standard space $\cC([-\delta,\delta],\hsig)$, see Proposition~\ref{prop local cauchy theory}. In particular, this local theory does not incorporate dispersive effects.\\  The fixed point procedure allows us to consider for every initial data $u_0 \in \hsigt$ the interval: 
\begin{equation}\label{I_max}
	I_{max}(u_0) := \bigcup_{I \in \mathfrak{I}} \hsp I
\end{equation}
where $\mathfrak{I}$ is the set formed by all the intervals $I$ containing 0 such that there exists a solution $u \in \cC(I,\hsigt)$ to~\eqref{pNLS} with initial data $u_0$. The local theory (see, here, Proposition~\ref{prop local cauchy theory}) ensures that $\mathfrak{I}$ is non-empty, and that $I_{max}(u_0)$ is open. It is the maximal interval on which the solution (emanating from $u_0$) exists. \\

\noindent \textbf{Deterministic globalization arguments :} Once the local Cauchy has been elaborated, one naturally wonders whether the solutions obtained are global (defined on $\R$). Proving that the equation is (deterministically) globally well-posed in $\hsigt$ consists in establishing that:
\begin{equation*}
	\forall u_0 \in \hsigt, \hspace{0.3cm} I_{max}(u_0) = \R
\end{equation*} 
Again by the fixed point algorithm, we know that there is the alternative:
\begin{align*}
	&\textnormal{If $I_{max}(u_0) \neq \R$, then:} & &\textnormal{ $\| u(t) \|_{\hsig} \to + \infty$, \hsp \hsp as $t \to \p I_{max}(u_0)$}
\end{align*}
where $u \in \cC(I_{max}(u_0), \hsig)$ is the (maximal) solution emanating from $u_0$. In other words, the $\hsigt$-norm of non-global solutions blow up in finite time. As a consequence, if one is able to prove that:
\begin{equation*}
	\forall t \in I_{max}(u_0), \hsp  \hsp \| u(t) \|_{\hsig} \leq f(t) 
\end{equation*}
for a certain continuous function $f : \R \to \R_+$, then one can conclude that $I_{max}(u_0) = \R$, meaning that $u$ exists for every time. \\
-- In the range of regularity $\sigma \geq 1$, one is able to establish such a bound using the conservation of the \textit{mass} and the \textit{Hamiltonian}:
\begin{align*}
	M(u) &:= \int_{\T} |u|^2 dx, & H(u)&:= \frac{1}{2} \int_\T |\p_x u|^2 dx + \frac{1}{p+1} \int_\T |u|^{p+1} dx
\end{align*} 
Indeed, these conservation laws imply that the $H^1(\T)$-norm of the solutions stays bounded, which, by a Grönwall argument using the Duhamel formulation~\eqref{Duhamel map}, yields an exponential bound for the $\hsigt$-norm of the solutions. Hence, from the criteria above, this means that the solution of~\eqref{pNLS} are global. It is also worth noting that this exponential bound can be improved into a polynomial one using more sophisticated method such as normal form reduction and the (upside-down) $I$-method (see~\cite{bourgain96_growth_sobolev_norms,oh_kwon_colliander_upside_down,sohinger2011_growth_sobolev_norms,staffilani97_growth_sobolev_norms}). \\  
-- In the range of regularity $\frac{1}{2} - \frac{2}{p-1} < \sigma < 1$, this standard globalization procedure is no longer applicable, because even though the mass is still finite and conserved, the full energy:
\begin{equation}\label{standard energy}
	E(u) := \frac{1}{2}\int_{\T} |u|^2 dx + \frac{1}{2} \int_\T |\p_x u|^2 dx + \frac{1}{p+1} \int_\T |u|^{p+1} dx
\end{equation} 
may be infinite (when $u_0 \in \hsigt \setminus H^1(\T)$), thereby making any hope of controlling the $H^1(\T)$-norm illusory. To overcome this obstacle, a technique -- the \textit{$I$-method} -- has been developed two decades ago by Colliander, Keel, Staffilani, Takaoka and Tao (see for instance in \cite{I_team_almost_conservation_laws}). The key idea is that even for $\sigma < 1$, the energy~\eqref{standard energy} plays a role. The method consists in considering a specific smoothing (Fourier multiplier) operator $I$ of order $1-\sigma$, which maps the $\hsigt$-solutions into $H^1(\T)$, and then use the (finite) modified energy $E(Iu)$, acting like an "almost conservation law", to obtain a bound on the $\hsigt$-norm of the solutions. In the end, the application of the $I$-method would allow to prove the global well-posedness for~\eqref{pNLS} up to $\sigma > \sigma_p$, for some $\sigma_p \in [\frac{1}{2}- \frac{2}{p-1},1)$. Moreover, it is worth noting that, in principle, this threshold $\sigma_p$ could be lowered by incorporating dispersive technologies (normal form reduction, Strichartz estimates, resonant decomposition) as it was performed for example in~\cite{bernier2023dynamics,bourgain2004remark,Li_Wu_Xu_global,schippaGWP}. For a pleasant introduction to the $I$-method we refer to~\cite{book_tzirakis} and~\cite{tao_dispersive}. \\

\noindent \textbf{Probabilistic globalization argument :} In this paper, we explore an other direction to globalize the local solutions of~\eqref{pNLS} adapted to random Gaussian initial data. Here the initial data take the form of the following random Fourier series ($s\in \R$):
\begin{equation}\label{Gaussian rva}
	u_0^\omega := \sum_{n \in \Z} \frac{g_n(\omega)}{\langle n \rangle^s} e^{inx}
\end{equation}
where $(g_n)_{n\in \Z}$ are independent complex standard Gaussian random variables on a probability space $(\Omega, \cF, \P)$. For $\sigma < s - \frac{1}{2}$, this series converges in the space $L^2(\Omega, \hsig)$, therefore defining a random variable valued in $\hsigt$ whose law, denoted $\mu_s$, defined a Gaussian (probability) measure on $\hsigt$:
\begin{equation*}
	\mu_s := \textnormal{law}(u_0^\omega) = (u_0^\omega)_\# \P
\end{equation*}
where the symbol $\#$ refers to the push-forward measure. For details about Gaussian measures, we refer to~\cite{kuo2006gaussian} and~\cite{bogachev1998gaussian}. It is well-know that almost surely:
\begin{align*}
	u_0^\omega &\in \bigcap_{r < s -\frac{1}{2}} H^r(\T)  & &\textnormal{and} & u_0^\omega & \notin H^{s-\frac{1}{2}}(\T)
\end{align*}
Since we are interested in initial data below the energy level, we take $s \leq \frac{3}{2}$, so that, again, the standard (deterministic) globalization argument by conservation of the energy is not applicable. Also, we take $s>1$ to be able to apply the local theory from Proposition~\ref{prop local cauchy theory}. Since we aim to construct a global flow to~\eqref{pNLS}, we are interested in the \textit{global well-posedness set}:
\begin{equation}\label{GWP set}
	\cG := \big\{ u_0 \in \hsigt \hsp : \hsp I_{max}(u_0) = \R \big\}
\end{equation}
where we recall that $I_{max}(u_0)$ is the maximal interval of existence of the solution generated by $u_0$, defined in~\eqref{I_max}. Let us briefly justify that $\cG$ is a Borel set. It is even a $\cG_\delta$ set. To see this, we first write:
\begin{equation*}
	\cG = \bigcap_{N \in \N} \hsp  \{u_0 \in \hsigt \hsp :\hsp [-N,N] \subset I_{max}(u_0) \},
\end{equation*}
and then, we observe that every set in the intersection is open. Indeed, by continuous dependence on the initial data (which is a consequence of the fixed point procedure to construct the solution), if $u_0$ generates a solution that lives on $[-N,N]$, then so do the solutions generated by initial data close enough to $u_0$. \\
We recall that for $\sigma \geq 1$, we have $\cG = \hsigt$. Below the energy level, we prove as a first result that:
\begin{thm}\label{thm GWP} Let $p \geq 5$ an odd integer. Let $\frac{3}{2} \geq s>s_p$, where $s_p$ is defined by:
	\begin{equation}\label{sp threshold}
		s_p = \frac{3}{2} - \frac{p}{4}(1 - \sqrt{1-\frac{8}{p^2}})
	\end{equation}
	Then, for $\sigma < s - \frac{1}{2}$ close to $s-\frac{1}{2}$, we have
	\begin{equation*}
		\mu_s(\cG^c) = 0
	\end{equation*}
\end{thm}
The pioneering work~\cite{bourgain1994periodic} initiated a general strategy -- the so-called Bourgain's invariant measure argument -- to prove analogous results with respect to \textit{invariant measures}, notably the \textit{Gibbs measure}; and it has since been applied for various models. Here, this method cannot be followed directly, since the Gaussian measures $\mu_s$ are not expected to be invariant. However, we will prove Theorem~\ref{thm GWP} in a similar spirit, but relying instead on the \textit{quasi-invariance} of these Gaussian measures (which will be combined with a deterministic Cauchy theory). In~\cite{forlano2022quasi}, this idea was explored in a more involved setting for the cubic fractional NLS, with a probabilistic Cauchy theory for low regularity Gaussian initial data (below $L^2(\T)$). \\

The study of quasi-invariance of Gaussian measures was initiated by Tzvetkov in~\cite{tzvetkov2015quasiinvariant}. Let us recall that considering the flow:
\begin{equation}\label{flow pnls}
	\Phi : (t,u_0) \longmapsto \Phi_t(u_0):=\textnormal{the solution of~\eqref{pNLS} with initial data $u_0$ at time $t$}
\end{equation}
we say that $\mu_s$ is quasi-invariant under the flow $\Phi$ if for every time $t\in \R$:
\begin{align*}
	\textnormal{law}(\Phi_t(u_0^\omega)) &\ll  \textnormal{law}(u_0^\omega), & & \textnormal{that is,} & (\Phi_t)_\# \mu_s &\ll \mu_s,
\end{align*}
where $\ll$ denotes absolute continuity of measures, meaning that for all Borel set $A$ in $\hsigt$:
\begin{equation*}
	\mu_s(A) = 0 \implies \mu_s( \Phi_t^{-1}A) = 0
\end{equation*}
Here, we encounter an obstruction because the knowledge of the local theory for~\eqref{pNLS} does not allow us a priori to consider, for a fixed time $t \in \R$, the object $\Phi_t(u)$ for $\mu_s$-almost every $u$. It then appears that the idea of using the quasi-invariance of $\mu_s$ precisely to construct the flow is a circular argument. For this reason, we consider first a truncated version of~\eqref{pNLS} for $N \in \N$:     
\begin{equation}\label{truncated equation}
	\begin{cases}
		i\p_t u + \p_x^2 u = \pi_N \left(|\pi_Nu|^{p-1}\pi_Nu \right), \hspace{0.2cm} (t,x) \in \R \times \T \\
		u|_{t=0} = u_0 
	\end{cases}
\end{equation}
where $\pi_N$ is the sharp Fourier projector onto frequencies of absolute value $\leq N$. We denote by $\Phi^N$ its flow (similarly defined as in~\eqref{flow pnls}), and we refer to it as the \textit{truncated flow}. This equation is interesting for few reasons: the solutions are global, they conserve the truncated energy:
\begin{equation}\label{standard energy E_N}
	E_N(u) := M(u)+ H(\pi_Nu),
\end{equation}
they share a common local Cauchy theory with~\eqref{pNLS} (see Proposition~\eqref{prop local cauchy theory}), and they approximate the solutions of~\eqref{pNLS} (see Proposition~\ref{prop approximation by the truncated flow}). Moreover, the Gaussian measures $\mu_s$ are quasi-invariant under the truncated flow (for every $s \in \R$), with an explicit formula for the Radon-Nikodym derivatives; more precisely, for every $t\in \R$:
\begin{align}\label{trsprt of gm}
	(\Phi^N_t)_\# \mu_s &= g_{s,N,t} d\mu_s, &  g_{s,N,t} &:= \exp \big( -\frac{1}{2}( \| \pi_N \Phi^N_{-t}(u)\|_{H^s(\T)}^2 - \| \pi_N u\|_{H^s(\T)}^2 )\big)
\end{align}
These properties are not proven here; however, the proofs can be found in~\cite{knezevitch2025qualitativequasiinvariancelowregularity,knez24transport}. With these ingredients, the crucial point to globalize the local solutions up to an arbitrary time $T>0$, consists in establishing a proper \textit{flow tail estimate} (see later in Lemma~\ref{lem flow tail estimate}) which take the form:
\begin{equation*}
	\mu_s\big( u_0 \hsp : \hsp \sup_{t \in [0,T]} \| \Phi_t^Nu_0 \|_{\hsig} > M \big) \leq f_s(T,M)
\end{equation*} 
where $f_s(T,M)$ is independent of $N$ and decays sufficiently fast to 0 as $M \to \infty$. This estimate is obtained through an iterative argument (see proof of Lemma~\ref{lem flow tail estimate}) -- in $\lfloor \frac{T}{\delta} \rfloor$ steps, with $\delta \sim M^{1-p}$ the appropriate local existence time  -- which, at every step $k$, yields the quantity:
\begin{align*}
	(\Phi^N_{k \delta})_\# \mu_s \big( \| u_0 \|_{\hsig} > M \big) &= g_{s,N,k \delta} \hsp d\mu_s \big( \| u_0 \|_{\hsig} > M \big) \\
	& \leq \| g_{s,N,k \delta} \|_{L^q(d\mu_s)} \mu_s \big( \| u_0 \|_{\hsig} > M \big)^{1-\frac{1}{q}}
\end{align*}
estimated with the Hölder inequality ($1<q<\infty$). Unfortunately, here, the (uniform in $N$) $L^q(d\mu_s)$-integrability for the Randon-Nikodym derivatives~\eqref{trsprt of gm} is out of reach, and to obtain such a bound, we need to add a suitable cutoff to $\mu_s$. This is an other central point in our analysis. Using ideas from~\cite{forlano2022quasi} and~\cite{tzvetkov2013gaussian}, we choose a renormalized-energy cutoff. Indeed, since the standard energy $E$, for which in our regime $s \leq \frac{3}{2}$ we have:
\begin{equation*}
	E(u_0) = + \infty \hspace{0.3cm} \textnormal{$\mu_s$-almost-surely,}
\end{equation*}
cannot be considered directly, we introduce instead a renormalization of it, which consists in removing a divergent term.  More precisely, introducing the truncated renormalizations (for $N \in \N$):
\begin{align}\label{renormalized energies}
	\cE_N (u) &:= M(u) + \big| H(\pi_N u) - \sigma_N \big| & & \textnormal{with:}  & \sigma_N := \frac{1}{2} \sum_{|n| \leq N} \frac{|n|^2}{\langle n \rangle^{2s}},
\end{align}
that are conserved by $\Phi^N$, since $\Phi^N$ conserves $E_N$ in~\eqref{standard energy E_N},
we have for $\frac{5}{4} < s \leq \frac{3}{2}$ and every $q \in [1,\infty)$ (see Proposition~\ref{prop cvgce renormalized energies}):
\begin{align*}
	&	\textnormal{$(\cE_N)_{N \in \N}$ is Cauchy in $L^q(d\mu_s)$, and therefore:} & & \exists \cE \in L^q(d\mu_s), \hsp \hsp \cE_N \tendsto{N} \cE \hsp \hsp \textnormal{in $L^q(d\mu_s)$}
\end{align*}
Hence, we will work with the \textit{cutoff Gaussian measures}:
\begin{align}\label{cutoff gm}
	\mu_{s,R,N} &:= \chi_R(\cE_N(u)) d\mu_s,  & \mu_{s,R} &:= \chi_R(\cE(u)) d\mu_s 
\end{align}
for $s \in (\frac{5}{4}, \frac{3}{2}]$, $R>0$, and where $\chi_R = \chi(\frac{.}{R})$, with $\chi : \R \to [0,1]$ smooth such that $\chi \equiv 1$ on $[-\frac{1}{2},\frac{1}{2}]$ and $\chi \equiv 0$ on $[-1,1]^c$. We will see the benefits of these cutoffs in Proposition~\ref{prop estimates with bounded renormalized energy}. It is however worth noting that we are limited here to the values of $s$ in $(\frac{5}{4},\frac{3}{2}]$. The main tools to obtain $L^q$-bounds (with respect to these measures) for the Radon-Nikodym derivatives will be a Poincaré-Dulac normal form reduction, suitable bounds on the \textit{resonant function} and the \textit{symmetrized derivative} (see Definition~\ref{def psi and omega}), and the Boué-Dupuis variational formula. \\

\noindent \textbf{Properties of the resulting flow :} Thanks to Theorem~\ref{thm GWP}, (recalling that $\cG$ is defined in~\eqref{GWP set}) we have for every time $t \in \R$ a well-defined flow to~\eqref{pNLS}:
\begin{equation*}
	\Phi_t \hsp : \hsp \cG \lra \cG
\end{equation*}
for which we prove the following properties:
\begin{thm}[Quantitative quasi-invariance]\label{thm quantitative quasiinv for the original flow}
	Let $\frac{3}{2} \geq s > s_p$. Then, for every $t\in \R$, there exists a non-negative measurable function $g_{s,t}$ such that:
	\begin{equation*}
		(\Phi_t )_\# \mu_s = g_{s,t} \hsp \mu_s
	\end{equation*}
	Furthermore, the renormalized energy  $\cE$ is invariant under the flow $\Phi$, meaning that for every $t \in \R$,
	\begin{equation*}
		\cE(\Phi_t u) = \cE(u), \hspace{0.3cm} \textnormal{$\mu_s$-almost everywhere}
	\end{equation*}
	Moreover, for every $R>0$, the cutoff Gaussian measure $\mu_{s,R}$ is transported by the flow as:
	\begin{equation*}
		(\Phi_t )_\# \mu_{s,R} = g_{s,t} \hsp \mu_{s,R}
	\end{equation*}
	and, for every $q \in [1,+\infty)$, there exists a constant $C_{s,R,q} > 0$ such that:
	\begin{equation}\label{Lq mu_s,R bound RND}
		\| g_{s,t}\|_{L^q(d\mu_{s,R})} \leq \exp \big( C_{s,R,q}(1+|t|)^A \big)
	\end{equation} 
	where $A>0$ is independent of $q$ and $R>0$.
\end{thm}
Note here that we see all the measures on $\hsigt$. For example, if $B$ is a Borel set of $\hsigt$,
\begin{equation*}
	(\Phi_t )_\# \mu_s(B) := (\Phi_t )_\# \mu_s(B \cap \cG) = \mu_s(\Phi_t^{-1}(B \cap \cG)) 
\end{equation*}
This theorem provides an improved version of the quasi-invariance of the Gaussian measures $\mu_s$, as it includes $L^q(d\mu_{s,R})$-bounds of the Radon-Nikodym derivatives. Similar results were obtained in~\cite{debussche2021quasi,forlano2025improvedquasiinvarianceresultperiodic,forlano_and_soeng2022transport,genovese2023transport,knez24transport,planchon2022modified}. We can derive from these bounds quantitative results on the solutions. For example, we are able to obtain a polynomial bound on the Sobolev norm of the solutions:
\begin{cor}\label{cor Growth of Sob norms}
	Let $\frac{3}{2} \geq s>s_p$. There exists $A>0$ such that for every $t \in \R$, we have for $\mu_s$-almost every $u_0$:
	\begin{equation}\label{polynomial growth of Sob norms}
		\| \Phi_t u_0 \|_{\hsigt} \lesssim_{s,u_0} (1+|t|)^A
	\end{equation}
	where $\sigma < s -\frac{1}{2}$ is closed to $s -\frac{1}{2}$.
\end{cor}

\noindent \textbf{Further remarks and comments:} \\
-- The threshold $s > s_p$ is technical, it comes from the proof of the $L^q$-bounds on the Radon-Nikodym derivatives in Proposition~\ref{prop RND in Lq uniformly in N}. Even though, the formula for $s_p$ is complicated, we have: 
\begin{align*}
	s_p &< \frac{3}{2} - \frac{1}{p} & &\textnormal{and asymptotically,} & s_p & \sim_{p \to \infty} \frac{3}{2} - \frac{1}{p}
\end{align*}
It is not excluded that this threshold could be improved. However, it is also important to note that there is still the limitation $s>\frac{5}{4}$ appearing in Proposition~\ref{prop cvgce renormalized energies} to consider the renormalized energy, which is crucial in our analysis.\\
-- Theorem~\ref{thm GWP} ensures that Gaussian initial data of regularity $H^{s_p - \frac{1}{2}+}(\T)$ generate global solutions of~\eqref{pNLS}. On the other hand, a deterministic approach using the $I$-method would in principle allow to prove the global-wellposedness in $H^{\sigma}(\T)$, with $\sigma > \sigma_p$, for some $\sigma_p < 1$. For example in the case $p=5$, the (deterministic) global well-posedness was obtained for all $\sigma > \frac{2}{5}$ in~\cite{Li_Wu_Xu_global} (see also~\cite{bernier2023dynamics}), which is beyond the scope of this paper. Instead of considering these two approaches independently, it could be interesting to explore whether the strengths of both methods can be leveraged together.\\
-- In the case $p=5$, Theorem~\ref{thm quantitative quasiinv for the original flow} completes a result established in~\cite{knez24transport}, where the quasi-invariance with $L^q$ bounds on the Radon-Nikodym derivatives with respect to (standard) energy-cutoff Gaussian measures is proven for the full range $s > \frac{3}{2}$. \\

\noindent \textbf{Acknowledgments:} The author would like to thank Chenmin Sun, Leonardo Tolomeo, and Nikolay Tzvetkov for suggesting this problem. The author is also grateful to Leonardo Tolomeo for an interesting discussion, and to Chenmin Sun and Nikolay Tzvetkov for their support and advice during the writing of this manuscript. This work is partially supported by the ANR project Smooth ANR-22-CE40-0017.

%% file: preliminaries.tex
In this section, we gather some preliminaries. First, we state a lemma which explains how the condition on $s$ in the theorems will appear. Second, we introduce quantities that will play a role later. Third, we construct our probabilistic toolbox.  

\begin{lem}\label{lem threshold quasiinv}
	Let $s > 0$ and $p \geq 5$. Then, 
	\begin{equation*}
		(3-2s)(2s + p - 3) < 2 \iff s > s_p
	\end{equation*}
	where $s_p$ is defined in~\eqref{sp threshold}.
\end{lem}

\begin{proof}
	We have,
	\begin{equation}\label{polyonmial cd on s}
		(3-2s)(2s + p - 3) < 2 \iff (3-2s)^2 -(3-2s) p + 2 > 0 
	\end{equation}
	Moreover, the polynomial $x \mapsto x^2 - px + 2$ is positive exactly when $x < \frac{1}{2}(p-\sqrt{p^2-8})$ or $x > \frac{1}{2}(p+\sqrt{p^2-8})$, so~\eqref{polyonmial cd on s} is satisfied if and only if:
	\begin{align*}
		s &> \frac{3}{2} - \frac{1}{4}(p - \sqrt{p^2-8}) = s_p & & \textnormal{or} & s &< \frac{3}{2} - \frac{1}{4}(p + \sqrt{p^2-8})
	\end{align*}
	Since $p \geq 5$, we have $\frac{3}{2} - \frac{1}{4}(p + \sqrt{p^2-8}) < 0$, and since we consider positive values of $s$, the condition on the right is never satisfied. This completes the proof.
\end{proof}

\subsection{The resonant function and the symmetrized derivatives}
In this paragraph, we introduce and study quantities that emerge in the Poincaré-Dulac normal reduction of Section~\ref{section Poincaré-Dulac normal form reduction and deterministic estimate}. In particular, we provide a convenient lower bound on the resonant function. Here we recall that $p \geq 5$ is an odd integer.   
 
\begin{defn}\label{def psi and omega}
	For $\vec{k}=(k_1,...,k_{p+1})\in \Z^{p+1}$, we define the quantities:
	\begin{align*}
		\psi_{2s}(\vec{k}) &:= \sum_{j=1}^{p+1} (-1)^{j-1}|k_j|^{2s}, & \Omega(\vec{k}) &:= \sum_{j=1}^{p+1} (-1)^{j-1} k_j^2
	\end{align*}
	called respectively the \textit{symmetrized derivative} (of order $2s$) and the \textit{resonant function}. \\
	Moreover, we consider:
		\begin{align*}
		\Psi_{2s}^{(0)}(\vec{k}) &:= \1_{\Omgz} \psi_{2s}(\vec{k}), & \Psi_{2s}^{(1)}(\vec{k}) &:= \1_{\Omgnz} \frac{\psi_{2s}(\vec{k})}{\Omgk}
	\end{align*}
\end{defn}

For the two following lemmas we use this notation:
\begin{notn}
	Given a set of frequencies $k_1,...,k_m \in \Z$, we denote by $k_{(1)},...,k_{(m)}$ a rearrangement of the $k_j's$ such that 
	\begin{equation*}
		|k_{(1)}| \geq |k_{(2)}| \geq ... \geq |k_{(m)}| 
	\end{equation*}
\end{notn} 

\begin{lem}\label{lemma psi estimate}
	Let $s>1$. For every $\vec{k}=(k_1,...,k_{p+1}) \in \Z^{p+1}$ such that:
	\begin{equation*}
		\sum_{j=1}^{p+1} (-1)^{j-1} k_j = 0
	\end{equation*} 
	we have,
	\begin{equation}\label{psi estimate}
		| \psi_{2s}(\vec{k}) | \leq C_s  k_{(1)} ^{2(s-1)} ( | \Omega(\vec{k}) |+ k_{(3)} ^2) 
	\end{equation}
	where the constant $C_s>0$ only depends on $s$ (and $p$). As a consequence, we have~: 
	\begin{equation}\label{psi estimate when Omg=0}
		| \Psi^{(0)}_{2s}(\vec{k}) | \leq C_s  k_{(1)} ^{2(s-1)} k_{(3)}^2
	\end{equation}
	and,
	\begin{equation}\label{psi/Omg estimate}
		| \Psi^{(1)}_{2s}(\vec{k}) | \leq C_s \big( k_{(1)} ^{2(s-1)} +  k_{(1)} ^{2(s-1)} \frac{k_{(3)}^2}{|\Omgk|} \big)
	\end{equation}
\end{lem}

These estimates are standard; for proofs, we refer for example to~\cite{sun2023quasi}, Lemma~4.1. We now state a lemma that provides a practical lower bound on the resonant function $\Omega$. This is a very slight adaptation of Lemma 6.1 from~\cite{knezevitch2025qualitativequasiinvariancelowregularity}, where a proof is given for $p=5$.

\begin{lem}\label{lem lower bound Omega when three high and the rest is low freq}
	There exists a constant $c_p>0$ such that for every $ (k_1,...,k_{p+1}) \in \Z^{p+1} \setminus \{0\}$ such that:
	\begin{align}\label{conditions linc and k4 ll k3}
		\sum_{j=1}^{p+1} (-1)^{j-1} k_j & = 0 & \textnormal{and}& & |k_{(4)}| & \leq \frac{1}{10p} |k_{(3)}| 
	\end{align} 
	we have: 
	\begin{equation}\label{lower bound Omg three high and the rest is low freq}
		|\Omgk| \geq c_p |k_{(1)}||k_{(3)}|
	\end{equation}
\end{lem}

\begin{rem}\label{rem Omgk = 0 implies N4 sim N3} 
	As a consequence, note that if $\Omgk = 0$ and $k_1-k_2+...-k_{p+1} = 0$ and $k_{(3)} \neq 0$, then we have $|k_{(3)}| < 10 p |k_{(4)}|$.
\end{rem}

\subsection{Probabilistic tools}

\begin{lem}\label{lem Y in W^sig,infty}
	Let $s \in \R$ and $\sigma < s-\frac{1}{2}$. Let $Y$ be the Gaussian random variable defined in~\eqref{Gaussian rva}. For $N \in \N$, set $Y_N:= \pi_N Y$. Then, for every $q \in [1,\infty)$, the sequence $(Y_N)_N$ converges to $Y$ in $L^q(\Omega,W^{\sigma,\infty})$. In particular,
	\begin{equation*}
		\sup_{N \in \N} \hsp \E\big[ \| Y_N \|^q_{W^{\sigma,\infty}} \big] \leq C_q < \infty
	\end{equation*}
\end{lem}

We refer to~\cite{forlano_and_soeng2022transport}, Lemma 5.3, for a proof of this lemma. 

\begin{lem}[Wiener Chaos estimate]\label{lem Wiener Chaos}
	Let $k \in \N$ and $c : \Z^k \ra \C$. Consider the following expression:
	\begin{equation*}
		S(\omega) := \sum_{(n_1,...,n_k)\in \Z^k} c(n_1,...,n_k) g^{\iota_1}_{n_1}(\omega)...g^{\iota_k}_{n_k}(\omega)
	\end{equation*} 
	where the $g_n$ are complex standard i.i.d Gaussians, and $\iota_j \in \{-,+\}$ is the complex conjugation or the identity whether $\iota_j = -$ or $\iota_j = +$ respectively. \\
	Suppose that $S\in L^2(\Omega)$, then, there exists $C>0$ such that for every $q \in [1,\infty)$:
	\begin{equation*}
		\| S \|_{L^q(\Omega)} \leq C q^{\frac{k}{2}} \| S\|_{L^2(\Omega)}
	\end{equation*}
\end{lem}
For a proof of this lemma, we refer to \cite{Simon+1974} (see also \cite{thomann2010gibbs} and \cite{tzvetkov2010construction}).\\

In this paper, an important element to prove $L^q$-bounds on the Radon-Nikodym derivatives~\eqref{trsprt of gm} is the Boué-Dupuis variational formula introduced in \cite{boue_dupuis98}. See also the related paper by Üstünel \cite{ustunel14}. This approach was adopted by Barashkov and Gubinelli in \cite{barashkov_gubinelli20}, and it has since been applied in the context of quasi-invariance in \cite{coe2024sharp,forlano2025improvedquasiinvarianceresultperiodic,forlano_and_soeng2022transport,forlano2022quasi}. More precisely, we will use the following simplified version of the Boué-Dupuis formula, which has been stated and proven in Lemma 2.6 of \cite{forlano2022quasi}. 
\begin{lem}[see Lemma 2.6 of \cite{forlano2022quasi}]\label{lem boue-dupuis}
	Let $s\in \R$, $N \in \N$, and $Y \sim \mu_s$ the Gaussian random variable defined in~\eqref{Gaussian rva}. Let $\cF:C^\infty(\T)\ra\R$ be a measurable function such that $\E[|\cF_-(\pi_N Y)|^p]~< \infty$, for some $p>1$, where $\cF_-$ denotes the negative part of $\cF$. Then,
	\begin{equation}\label{ineq boue dupuis}
		\log \int e^{\cF(\pi_N u)} d\mu_s	\leq \E \big[\sup_{V \in H^s} \{ \cF(\pi_N Y + \pi_N V) - \frac{1}{2}\norm{V}^2_{H^s} \} \big]
	\end{equation}
\end{lem}

%% file: Quantitative_quasiinv_and_GWP.tex
In this section, we prove Theorem~\ref{thm GWP}. To do so, we use: a local theory that is applicable to~\eqref{pNLS} and~\eqref{truncated equation} for all $N \in \N$ (every parameter in Proposition~\ref{prop local cauchy theory} will be independent of $N \in \N$), a suitable approximation of the original flow $\Phi$ by the truncated flow $\Phi^N$ (see Proposition~\ref{prop approximation by the truncated flow}), and a flow tail estimate (see Lemma~\ref{lem flow tail estimate}). Note that we assume here the crucial bounds~\eqref{RND in Lq uniformly in N} which are proved later in Section~\ref{section Uniform Lq integrability of the Radon-Nikodym derivatives} (more precisely in paragraph~\ref{subsection proof of the Lq integrability}).

\subsection{Local Cauchy theory and approximation}
We first recall the local well-posedness of~\eqref{pNLS} in $\hsigt$, $\sigma > \frac{1}{2}$. Even though we already know that~\eqref{truncated equation} is globally well-posed, the fact that it verifies the same local theory than~\eqref{pNLS} is important. We consider the Duhamel map $\Gamma_{u_0}$ in~\eqref{Duhamel map}, and we similarly define $\Gamma^N_{u_0}$ replacing $|u|^{p-1}u$ by $\pi_N (|\pi_Nu|^{p-1}\pi_Nu)$.
\begin{prop}\label{prop local cauchy theory} Let $\sigma > \frac{1}{2}$. There exist universal constants $C_0>0$ and $c_1>0$ such that for every $R_0>0$, denoting $R := C_0 R_0$ and $\delta := c_1 R_0^{1-p}$, we have that for every $\| u_0 \|_{\hsig} \leq R_0$, and every $N \in \N $, the maps:
	\begin{align*}
			\Gamma_{u_0}&: \hsp B_R(\delta) \to B_R(\delta) & & \textnormal{and,} & \Gamma^{N}_{u_0}: \hsp &B_R(\delta) \to B_R(\delta)
	\end{align*} 
	are all contractions, with the same universal contraction coefficient $\gamma \in (0,1)$, and where $B_R(\delta)$ is the closed centered ball of radius $R$ in $\cC([-\delta, \delta],\hsigt)$.
\end{prop} 

\begin{prop}\label{prop approximation by the truncated flow}
	Let $K \subset \hsigt$ a compact set, and let $I \subset \R$ a compact interval containing 0 such that every solution of~\eqref{pNLS} generated by an initial data in $K$ lives on $I$. Then,
	\begin{equation}\label{approx by the truncated flow}
		\sup_{u_0 \in K} \sup_{t \in I} \| \Phi_t u_0 - \Phi^N_t u_0 \|_{\hsig} \tendsto{N \to \infty } 0
	\end{equation}
\end{prop}
Proofs of the two previous propositions can be found in~\cite{knez24transport,knezevitch2025qualitativequasiinvariancelowregularity}.

\subsection{Transport of cutoff Gaussian measures} We know that the truncated flow $\Phi^N$ transports the Gaussian measures as in~\eqref{trsprt of gm}; thanks to the fact that the renormalized energy $\cE_N$ is preserved by $\Phi^N$, we deduce that it transports the cutoff Gaussian measures~\eqref{cutoff gm} similarly:

\begin{prop}\label{prop trspt cutoff gm}
	Let $N \in \N$, $t \in \R$, and $R>0$. Then, 
	\begin{equation}\label{trspt cutoff gm}
		(\Phi^N_t)_\# \mu_{s,R,N} = g_{s,N,t} \hsp d\mu_{s,R,N}
	\end{equation}
\end{prop}
The proof is standard, but we provide it nonetheless to emphasize the importance of the conservation of the cutoff $\chi_R(\cE_N(u))$ by $\Phi^N$. First we need the following lemma:

\begin{lem}\label{lem trspt density meas}
	Let $X$,$Y$ be two topological spaces, and $\Upsilon : X \to Y$ an homeomorphism. We consider a measure $\mu : \cB(X) \to [0,\infty]$ (on the Borel $\sigma$-algebra of $X$) and a non-negative measurable function $f : X \to [0,\infty]$. Then, the density measure $f d\mu$ is transported by $\Upsilon$ as:
	\begin{equation*}
		\Upsilon_\# (f d\mu) = (f \circ \Upsilon^{-1}) \Upsilon_\# \mu 
	\end{equation*}
\end{lem}

Let us briefly recall the proof of this lemma:

\begin{proof}
	For $\psi : Y \to [0,\infty]$ a non-negative measurable function, we have by definition of a push-forward measure:
	\begin{equation*}
		\int_Y \psi(v) 	d\Upsilon_\# (f d\mu)(v) = \int_X \psi(\Upsilon(u)) f(u) d\mu(u) = \int_Y \psi(v) (f \circ \Upsilon^{-1} )(v) d\Upsilon_\# d\mu(v)
	\end{equation*}
	which indeed implies that $\Upsilon_\# (f d\mu) =(f \circ \Upsilon^{-1}) \Upsilon_\# \mu $. 
\end{proof}

Next, we move on to the proof of Proposition~\ref{prop trspt cutoff gm}.

\begin{proof}
	Since $\Phi^N_t$ is bijective, with inverse $\Phi^N_{-t}$, Lemma~\ref{lem trspt density meas} gives:
	\begin{equation*}
		(\Phi^N_t)_\# \mu_{s,R,N} = (\Phi^N_t)_\# \big( \chi_R(\cE_N(u)) d\mu_s \big) = \chi_R(\cE_N(\Phi^N_{-t}u))(\Phi^N_t)_\# \mu_s
	\end{equation*}
	By~\eqref{trsprt of gm} and the conservation $\cE_N(\Phi^N_{-t}u) = \cE_N(u)$, we deduce that:
	\begin{equation*}
		(\Phi^N_t)_\# \mu_{s,R,N} = \chi_R(\cE_N(u))g_{s,N,t} d\mu_s = g_{s,N,t} d\mu_{s,R,N} 
	\end{equation*}
	as desired.
\end{proof}

\begin{rem}
	The fact that the cutoff in~\eqref{cutoff gm} is conserved by the flow is a crucial element in our analysis. In the context of quasi-invariance, it is also pertinent to consider cutoffs on the $\hsigt$-norm ($\sigma < s-\frac{1}{2}$) (see for instance~\cite{coe2024sharp,knezevitch2025qualitativequasiinvariancelowregularity,sun2023quasi}), considering the measures $\gamma_{s,R} := \1_{B_R^{\hsig}} \mu_s$, which are transported by the truncated flow as:
	\begin{equation*}
		(\Phi^N_t)_\# \gamma_{s,R} =   g_{s,t,N} \1_{\Phi^N_t(B_R)} d\mu_s
	\end{equation*}
	We see that the cutoff on $B_R$ is transported to the cutoff on $\Phi^N_t(B_R)$. This makes difficult any iterative argument as in the next paragraph, especially when we do not have any control on the balls $\Phi^N_t(B_R)$.
\end{rem}

\subsection{The globalization argument}
In this paragraph we prove Theorem~\ref{thm GWP}. For convenience, we assume here the propositions~\ref{prop cvgce renormalized energies} and~\ref{prop approx of rho by rhoN} on the cutoff Gaussian measures, see Section~\ref{section Renormalized energy and associated cutoff Gaussian measures} for the proofs; and the following quantitative bounds on the Radon-Nikodym derivatives, see Section~\ref{section Uniform Lq integrability of the Radon-Nikodym derivatives} for the proof.
\begin{prop}\label{prop RND in Lq uniformly in N}
	Let $\frac{3}{2} \geq s>s_p$, where $s_p$ is defined in~\eqref{sp threshold}. There exists a constant $A>0$, such that for every $N \in \N$, $R>0$, $t \in \R$: 
	\begin{equation}\label{RND in Lq uniformly in N}
		\| g_{s,N,t}\|_{L^q(d\mu_{s,R,N})} \leq \exp \big( C_{s,R,q}(1+|t|)^A \big)
	\end{equation}
	where the constant $C_{s,R,q}>0$ is independent on $N$.
\end{prop}
Next, we decompose the proof of Theorem~\ref{thm GWP} in two lemmas; then, the final argument is given at the end of the present paragraph.
\begin{lem}[Flow tail estimate]\label{lem flow tail estimate} Let $T>0$ and $\sigma < s - \frac{1}{2}$ close to $s - \frac{1}{2}$. Then,
	\begin{equation}\label{flow tail estimate}
		\mu_{s,R,N}\big(u_0 \hsp : \hsp \sup_{|t| \leq T} \| \Phi_t^Nu_0 \|_{\hsig} > M \big) \leq e^{C_{s,R} (1+T)^A} e^{-\alpha M^2}
	\end{equation}
	for some constant $\alpha>0$.
\end{lem}

\begin{proof}
	Let us consider the universal constants $C_0,c_1 >0$ from Proposition~\ref{prop local cauchy theory}. Introducing $\delta = c_1 (\frac{M}{C_0})^{1-p}$, we know that for every initial data $v_0 \in \hsig$,
	\begin{equation*}
	 \| v_0 \|_{\hsig} \leq \frac{M}{C_0} \implies \sup_{-\delta \leq t \leq \delta} \| \Phi^N_t v_0 \|_{\hsig} \leq M
	\end{equation*}
	Thus, using the additivity of the flow, and denoting $m:= \lfloor \frac{T}{\delta}\rfloor$, we have:
	\begin{equation*}
		\bigcap_{k=-m}^{m} \big\{ u_0 \hsp : \hsp \| \Phi^N_{k \delta} u_0\|_{\hsig} \leq \frac{M}{C_0} \big\} \subset \big\{ u_0 \hsp : \hsp \sup_{ |t| \leq T} \| \Phi^N_t \|_{\hsig} \leq M \big\}
	\end{equation*} 
	As a consequence,
	\begin{align*}
		&\mu_{s,R,N}\big(u_0 \hsp : \hsp \sup_{|t| \leq T} \| \Phi_t^Nu_0 \|_{\hsig} > M \big) \leq \sum_{k=-m}^m \mu_{s,R,N}\big( u_0 \hsp : \hsp \| \Phi^N_{k \delta} u_0\|_{\hsig} > \frac{M}{C_0} \big)  \\
		& = \sum_{k=-m}^m (\Phi^N_{k \delta})_\#\mu_{s,R,N}\big( \|  u_0\|_{\hsig} > \frac{M}{C_0} \big) \leq  \sum_{k=-m}^m \| g_{s,N,k\delta} \|_{L^2(d\mu_{s,R,N})} \mu_{s,R,N}\big(\|  u_0\|_{\hsig} > \frac{M}{C_0} \big)^{\frac{1}{2}}
	\end{align*}
	Now, on the one hand, $\| g_{s,N,k\delta} \|_{L^2(d\mu_{s,R,N})} \leq e^{C_{s,R} (1+T)^A}$ by~\eqref{RND in Lq uniformly in N}; and on the other hand, we have the standard Gaussian tail estimate (see for example Fernique's theorem in~\cite{bogachev1998gaussian} or~\cite{kuo2006gaussian}):
	\begin{align*}
		\mu_{s,R,N}\big(\|  u_0\|_{\hsig} > \frac{M}{C_0} \big) \leq \mu_s\big(\|  u_0\|_{\hsig} > \frac{M}{C_0} \big)  \lesssim e^{-\alpha M^2}
	\end{align*}
	for some constant $\alpha > 0$. Using this in the estimate above yields:
	\begin{align*}
		\mu_{s,R,N}\big(u_0 \hsp : \hsp \sup_{|t| \leq T} \| \Phi_t^Nu_0 \|_{\hsig} > M \big) &\lesssim  (\frac{T}{\delta} + 1) e^{C_{s,R} (1+T)^A} e^{-\alpha M^2} \\
		&\lesssim (1+T)e^{C_{s,R} (1+T)^A} (1+M)^{p-1}e^{-\alpha M^2}
	\end{align*}
	Hence, taking $C_{s,R}>0$ slightly larger and $\alpha > 0$ slightly smaller leads to~\eqref{flow tail estimate}. 
\end{proof}

\begin{lem}\label{lem existence up to T on compact sets}
	Let $T>0$ and $K \subset \hsigt$ a compact set, where $\sigma < s-\frac{1}{2}$ close to $s-\frac{1}{2}$. There exists a Borel subset $\Sigma_{K,T}$ of $K$, with full $\mu_s$-measure in $K$, such that every solutions generated by an initial data in $\Sigma_{K,T}$ lived on $[-T,T]$. In other words,
	\begin{align*}
	\mu_s( K \setminus \Sigma_{K,T}) &= 0, & & \textnormal{and:} & \forall u_0 \in \Sigma_{K,T}&, \hspace{0.3cm} I_{max}(u_0) \supset [-T,T]
	\end{align*} 
	where we recall that $I_{max}(u_0)$ is the maximal interval on which the solution of~\eqref{pNLS} with initial data $u_0$ exists, see~\eqref{I_max}.
\end{lem}

\begin{proof} Without loss of generality, we work on the interval $[0,T]$. We introduce an increasing sequence $\{R_l\}_{l \in  \N}$ of positive numbers tending to infinity. Our strategy consists in constructing first, for every $R_l$, and all integers $M$ such that $K \subset B_M^{\hsig}$, Borel subsets $\Sigma_{K,T,M,R_l}$ of $K$, such that:
	\begin{align*}
		\lim_{M \to \infty} \mu_{s,R_l}(K \setminus \Sigma_{K,T,M,R_l}) &= 0 & &\textnormal{and,} & \forall u_0 \in \Sigma_{K,T,M,R_l}&: \hspace{0.3cm} I_{max}(u_0) \supset [0,T] 
	\end{align*}
	Indeed, if we do so, then the following Borel subset of $K$:
	\begin{equation*}
		\Sigma_{K,T} := \bigcup_{l,M} \Sigma_{K,T,M,R_l}
	\end{equation*}
	would satisfy:
	\begin{align*}
		&\mu_s(K \setminus \Sigma_{K,T}) = \mu_s\big( \bigcap_{l,M} K \setminus  \Sigma_{K,T,M,R_l} \big) = \lim_{l' \to \infty} \mu_{s,R_{l'}} \big( \bigcap_{l,M} K \setminus  \Sigma_{K,T,M,R_l} \big) \\
		&\leq  \liminf_{l' \to \infty} \mu_{s,R_{l'}} \big( \bigcap_{M} K \setminus  \Sigma_{K,T,M,R_{l'}} \big) \leq \liminf_{l' \to \infty} \big[ \liminf_{M \to \infty } \mu_{s,R_{l'}} (K \setminus \Sigma_{K,T,M,R_{l'}}) \big] =0
	\end{align*}
	Therefore, this would complete the proof.\\
	
	In what follows, we construct such sets $\Sigma_{K,T,M,R_l}$. Let then $l\in \N$, and let $M$ be an integer such that $K \subset B^{\hsig}_M$. We consider $\delta \sim (1+M)^{p-1}$ the local existence time from Proposition~\ref{prop local cauchy theory} associated to initial data of size smaller than $M+1$. We denote $m := \lfloor \frac{T}{\delta} \rfloor$. Furthermore, thanks to Proposition~\ref{prop approx of rho by rhoN}, we introduce an integer $N_0$ such that:
	\begin{equation}\label{mu_Rl - mu_Rl,N}
		N \geq N_0 \implies \forall A \in \cB(\hsig), \hspace{0.3cm} |\mu_{s,R_l}(A) - \mu_{s,R_l,N}(A)| \leq \frac{1}{mM}
	\end{equation}
	-- Thanks to the approximation by the truncated flow (see Proposition~\ref{prop approximation by the truncated flow}), we know that there exists $N_1 \geq N_0$ such that:
	\begin{equation*}
		N \geq N_1 \implies \sup_{u_0 \in K} \sup_{0 \leq t \leq \delta} \| \Phi_t u_0 - \Phi^N_t u_0 \|_{\hsig}  \leq 1
	\end{equation*}
	Hence, introducing the set:
	\begin{equation*}
		K_1 := K \cap \big\{  \sup_{0 \leq t \leq \delta} \| \Phi^{N_1}_t u_0 \|_{\hsig} \leq M \big\}
	\end{equation*}
	 we deduce that for every $u_0 \in K_1$, $\| \Phi_\delta u_0  \|_{\hsig}  \leq M + 1$, which ensures by Proposition~\ref{prop local cauchy theory} that the solution generated by $u_0$ lives on $[0,2\delta]$, that is $I_{max}(u_0) \supset [0,2\delta]$. Next, since $K_1$ is compact (it is the intersection of the compact $K$ with a closed set), we can apply Proposition~\ref{prop approximation by the truncated flow}  as above (but now with $K_1$ and $[0,2 \delta]$), and find $N_2\geq N_0$ such that every initial data in:
	 \begin{equation*}
	 	 K_2 :=K \cap \big\{  \sup_{0 \leq t \leq \delta} \| \Phi^{N_1}_t u_0 \|_{\hsig} \leq M \big\} \cap \big\{  \sup_{0 \leq t \leq 2\delta} \| \Phi^{N_2}_t u_0 \|_{\hsig} \leq M \big\}
	 \end{equation*}
	 generates a solution that lives on $[0,3\delta]$. Thus, proceeding by induction (recalling that $m = \lfloor \frac{T}{\delta} \rfloor$), we are able to construct $N_m,...,N_1$, all greater that $N_0$, such that every initial data in the set:
	 \begin{equation*}
	 	\Sigma_{K,T,M,R_l} :=  \bigcap_{k=1}^{m} K \cap  \big\{ \sup_{0 \leq t \leq k\delta} \| \Phi^{N_k}_t u_0 \|_{\hsig} \leq M  \big\}
	 \end{equation*}
	  yields a solution on $[0,T]$. Note that this set depends on $R_l$ since we chose every $N_k$ greater than $N_0$, and that $N_0$ itself depends on $R_l$ (see~\eqref{mu_Rl - mu_Rl,N}). \\
	  
	  -- Next, we estimate the $\mu_{s,R_l}$-measure of this set. Since every $N_k$ is greater than $N_0$, we can write (thanks to~\eqref{mu_Rl - mu_Rl,N}):
	  \begin{align*}
	  	\mu_{s,R_l}(K \setminus \Sigma_{K,T,M,R_l}) &\leq \sum_{k=1}^m \mu_{s,R_l}\big(\sup_{0 \leq t \leq k\delta} \| \Phi^{N_k}_t u_0 \|_{\hsig} > M\big) \\
	  	& \leq \frac{1}{M} + \sum_{k=1}^m \mu_{s,R_l,N_k}\big(\sup_{0 \leq t \leq k\delta} \| \Phi^{N_k}_t u_0 \|_{\hsig} > M\big)
	  \end{align*}
	 Applying now~\eqref{flow tail estimate}, leads to:
	  \begin{equation*}
	   \mu_{s,R_l}(K \setminus \Sigma_{K,T,M,R_l})	\leq \frac{1}{M} +  C_{s,R_l,T} M^{p-1} \hsp e^{-\alpha M^2}
	  \end{equation*}
	  which indeed tends to $0$ as $M \to \infty$. This completes the proof.
\end{proof}

\begin{proof}[Proof of Theorem~\ref{thm GWP}]
	First, we recall that $\sigma < s - \frac{1}{2}$ is fixed and close to $s-\frac{1}{2}$. Then, we consider $\sigma' \in (\sigma ,s-\frac{1}{2})$. For every $n \in \N$, we introduce $B^{H^{\sigma'}}_n$, the  closed centered ball of radius $n$ in $H^{\sigma'}(\T)$, which is compact in $\hsigt$. Moreover, we introduce an increasing sequence $\{T_k\}_{k \in \N}$ of positive time tending to infinity. Applying Lemma~\ref{lem existence up to T on compact sets}, we consider for every $n,k \in \N$:
	\begin{align*}
		\Sigma_{n,k} &:= \Sigma_{B^{H^{\sigma'}}_n, T_k} & &\textnormal{and then we introduce:} & \Sigma &:= \bigcap_{k \in \N} \bigcup_{n \in \N} \Sigma_{n,k}
	\end{align*}
	Let us now observe that if $u_0 \in \Sigma$, then for every $k\in\N$, $I_{max}(u_0) \subset [-T_k, T_k]$. It means that $u_0$ generates a global solution of~\eqref{pNLS}. In other words, $\Sigma$ is a subset of the global well-posedness set $\cG$ (defined in~\eqref{GWP set}). Next, we prove that $\Sigma$ is of full $\mu_s$-measure. It is sufficient to prove that for every $k$, $\cup_n \Sigma_{n,k}$ is of full $\mu_s$-measure. Since we know that $H^{\sigma'}(\T)$ is of full $\mu_s$-measure, we can write:
	\begin{align*}
		\mu_s \big( \hsig \setminus \bigcup_{n \in \N} \Sigma_{n,k}\big) &= 		\mu_s \big( H^{\sigma'} \setminus \bigcup_{n \in \N} \Sigma_{n,k}\big) = \lim_{n' \to \infty} \mu_s \big( B^{H^{\sigma'}}_{n'} \setminus \bigcup_{n \in \N} \Sigma_{n,k} \big) \\
		& \leq \liminf_{n' \to \infty} \mu_s( B^{H^{\sigma'}}_{n'} \setminus \Sigma_{n',k} ) = 0
	\end{align*}
	As a consequence, $\mu_s(\Sigma^c) = 0$, and a fortiori $\mu_s(\cG^c) = 0$. This concludes the proof.
\end{proof}

%% file: Quasiinv_for_the_original_flow.tex
This section is dedicated to the proof of Theorem~\ref{thm quantitative quasiinv for the original flow} and its application to the growth of Sobolev norms in Corollary~\ref{cor Growth of Sob norms}. Again we assume~\eqref{RND in Lq uniformly in N} and Proposition~\ref{prop cvgce renormalized energies}, which are proven in sections~\ref{section Uniform Lq integrability of the Radon-Nikodym derivatives} and~\ref{section Renormalized energy and associated cutoff Gaussian measures} respectively.

\subsection{Proof of the quasi-invariance} In the previous section, we have proven that the flow was global for $\mu_s$-every initial data. Now, we prove that $\mu_s$ is quasi-invariant along this flow, with quantitative bounds on the Radon-Nikodym derivative. Moreover, we prove that the renormalized energy (defined as a limit in $L^q(d\mu_s)$ of truncated renoramlized energies) is indeed preserved by the flow. We rely  on the following classical lemma from functional analysis.

\begin{lem}\label{lem extension and cvgce of bdd linear forms}
	Let $E$ be a Banach space. Let $(L_n)_{n \in \N}$ be a sequence of linear forms that are uniformly bounded, in the sense that:
	\begin{equation*}
		\sup_{n\in \N} \| L_n \|_{E \to \R} < +\infty
	\end{equation*}
	Suppose that for a dense linear subset $D$ of $E$, we have:
	\begin{equation*}
		\forall x \in D, \hspace{0.3cm} \lim_{n \to \infty} L_n(x) =: L(x) \in \R \hspace{0.2cm} \textnormal{exists}
	\end{equation*}
	Then, the map $L : D \to \R$ is linear and extends continuously, and uniquely, in a linear form (still denoted $L$):
	\begin{align*}
		L &: E \to \R, & & \textnormal{such that:} & \| L \|_{E \to \R} \leq \sup_{n \in \N}  \| L_n \|_{E \to \R}
	\end{align*}
	Moreover, for every $x\in E$, $L_n(x) \to L(x)$.
\end{lem}

\begin{proof}[Proof of Theorem~\ref{thm quantitative quasiinv for the original flow}] Let $q \in (1,\infty)$ and $q' \in (1,\infty)$ its conjugate exponent. We want to apply Lemma~\ref{lem extension and cvgce of bdd linear forms} with $E$ the Banach space $L^{q'}(d\mu_s)$, whose dual is identified to $L^q(d\mu_s)$, and $D$ the linear subspace of bounded continuous function on $\hsigt$ (which is indeed dense in $L^{q'}(d\mu_s)$).\\
	
	-- Let $R>0$. Let $\psi : \hsig \to \R$ a bounded continuous function. Then, for $N \in \N$, we have from the definition of the Radon-Nikodym derivative~\eqref{trsprt of gm} and the invariance of $\cE_N$ by $\Phi^N_t$,
	\begin{equation}\label{int psi chi g_s,N}
	\int \psi(u) \chi_R(\cE_N(u)) g_{s,N,t} d\mu_s = \int \psi(\Phi^N_t u) \chi_R(\cE_N(u)) d\mu_s 
	\end{equation}
	Thanks to both~\eqref{approx by the truncated flow} and~\eqref{cvgce in Lq cutoff}, which in particular ensure that $\Psi(\Phi^N_t u)$ and $\chi_R(\cE_N(u))$ converge in $L^2(d\mu_s)$ to $\psi(\Phi_t(u))$ and $\chi_R(\cE(u))$ respectively, passing to the limit $N \to \infty$ in the equality above leads to:
	\begin{equation}\label{lim int psi chi g_s,N}
			\lim_{N \to \infty} \int \psi(u) \chi_R(\cE_N(u)) g_{s,N,t} d\mu_s = \int \psi(\Phi_t u) \chi_R(\cE(u)) d\mu_s = \int \psi(u) d(\Phi_t)_\# d\mu_{s,R} 
	\end{equation}
	Moreover, thanks to~\eqref{RND in Lq uniformly in N}, we know that that the sequence $\chi_R(\cE_N(u)) g_{s,N,t}$ is uniformly bounded in $L^q(d\mu_s)$. Hence, with~\eqref{int psi chi g_s,N} and~\eqref{lim int psi chi g_s,N}, we can appeal to Lemma~\ref{lem extension and cvgce of bdd linear forms}, and use the standard isomorphism $L^q(d\mu_s) \simeq (L^{q'}(d\mu_s))'$, to conclude that there exists $g_{s,t,R} \in L^q(d\mu_s)$ such that:
 	\begin{align}\label{RND depending on R}
 		(\Phi_t)_\#\mu_{s,R} &= g_{s,t,R} d\mu_s  & &\textnormal{with:} & \|  g_{s,t,R}\|_{L^q(d\mu_s)} \leq \exp \big( C_{s,R,q}(1+|t|)^{A} \big)
 	\end{align}
 	and,
 	\begin{equation}\label{weak cvgce densities}
 		 \chi_R(\cE_N(u)) g_{s,N,t} \lra g_{s,t,R}, \hspace{0.4cm} \textnormal{weakly in $L^q(d\mu_s)$}
 	\end{equation}
 	Since $R>0$ is arbitrary, we deduce that $(\Phi_t)_\# \mu_s$ is absolutely continuous with respect to $\mu_s$; indeed, if a Borel set $A$ is such that $\mu_s(A)=0$, then:
 	\begin{equation*}
 		(\Phi_t)_\# \mu_s (A) = \lim_{R \to \infty}(\Phi_t)_\# \mu_{s,R} (A) \leq \liminf_{R \to \infty}  \|  g_{s,t,R}\|_{L^q(d\mu_s)} \mu_s(A)^{\frac{1}{q'}} = 0
 	\end{equation*}
 	Invoking now the Radon-Nikodym theorem, there exists a non-negative measurable function $g_{s,t}$ such that:
 	\begin{equation}\label{Phi mu_s = g_s,t mu_s} 
 		(\Phi_t)_\# \mu_s = g_{s,t} \mu_s
 	\end{equation}
 	-- Let us now prove that $\cE(u) = \cE(\Phi_tu)$ $\mu_s$-almost everywhere. Again, we consider $\psi : \hsig \to \R$ a bounded continuous function, and we write:
 	\begin{align*}
 		&\int \psi(u) \cE(\Phi_t(u)) d\mu_{s,R} = \int \psi(\Phi_{-t}(u)) \cE(u) d (\Phi_t)_\# \mu_{s,R} \\
 		& =  \int \psi(\Phi_{-t}(u)) \cE(u) g_{t,s,R} d\mu_s 
 		= \lim_{N \to \infty} \int \psi(\Phi^N_{-t}(u)) \cE_N(u)   \chi_R(\cE_N(u)) g_{s,N,t} d\mu_s
 	\end{align*} 
 	This limit is justified thanks to the weak convergence in $L^q(d\mu_s)$ in~\eqref{weak cvgce densities} and the convergence in $L^q(d\mu_s)$ of $\psi(\Phi^N_{-t}(u)) \cE_N(u) $ to $\psi(\Phi_{-t}(u)) \cE(u)$ (see the propositions~\ref{prop approximation by the truncated flow} and~\ref{prop cvgce renormalized energies}). Using now the invariance of $\cE_N$ by the truncated flow, and again Proposition~\ref{prop cvgce renormalized energies}, we can continue the equality above as:  
 	\begin{align*}
 		\int \psi&(u) \cE(\Phi_t(u)) d\mu_{s,R} =  \lim_{N \to \infty} \int \psi(u) \cE_N(u)  \chi_R(\cE_N(u)) d\mu_s \\
 		& = \int \psi(u) \cE(u) \chi_R(\cE(u)) d\mu_s = \int \psi(u) \cE(u)d\mu_{s,R} 
 	\end{align*} 
 	Since $\psi$ is arbitrary, we can deduce that $\cE(\Phi_t(u)) = \cE(u)$ $\mu_{s,R}$-almost everywhere; and then, since $R$ is also arbitrary, we can conclude that this equality holds $\mu_s$-almost everywhere.\\
 	
 	-- Finally, thanks to~\eqref{Phi mu_s = g_s,t mu_s}, the invariance of $\cE$ by the flow, and Lemma~\ref{lem trspt density meas}, we have for every $R>0$: 
 	\begin{equation*}
 		(\Phi_t)_\# \mu_{s,R} = (\Phi_t)_\# \big( \chi_R(\cE(u)) d\mu_s\big) = \chi_R(\cE(\Phi_{-t}u)) g_{s,t} d\mu_s = \chi_R(\cE(u)) g_{s,t} d\mu_s = g_{s,t} d\mu_{s,R}
 	\end{equation*}
 	Using~\eqref{RND depending on R} and~\eqref{weak cvgce densities}, we also obtain:
 	\begin{equation*}
 		\|  \chi_R(\cE(u)) g_{s,t}\|_{L^q(d\mu_s)} \leq \exp \big( C_{s,R,q}(1+|t|)^{M_0} \big)
 	\end{equation*}
 	and,
 	\begin{equation*}
 		\chi_R(\cE_N(u)) g_{s,N,t} \lra  \chi_R(\cE(u)) g_{s,t} \hspace{0.4cm} \textnormal{weakly in $L^q(d\mu_s)$}
 	\end{equation*}
	This completes the proof.
\end{proof}

\subsection{Growth of Sobolev norms} Here, we prove Corollary~\ref{cor Growth of Sob norms} by adapting Bourgain's invariant argument. In the context of invariant Gibbs measures, one may expect logarithmic bounds on the Sobolev norms of the solutions. However, for quasi-invariant Gaussian measures, the Radon-Nikodym derivatives depend on time, and we need to control at every time its $L^q(d\mu_{s,R})$-norm. Here, the exponential control in~\eqref{Lq mu_s,R bound RND} yields polynomial bounds on the Sobolev norms of the solutions. The strategy of the following proof is not new (see for example~\cite{hofer2024growthsobolevnormsperiodic}).

\begin{proof}[Proof of Corollary~\ref{cor Growth of Sob norms}] For $T>0$, proceeding as in the proof of~\eqref{flow tail estimate}, we have the flow tail estimate:
	\begin{equation*}
	\mu_{s,R}\big(u_0 \hsp : \hsp \sup_{|t| \leq T} \| \Phi_tu_0 \|_{\hsig} > M \big) \leq e^{C_{s,R} (1+T)^A} e^{-\alpha M^2}
	\end{equation*}
	for any $M>0$. As a consequence, for $B > \frac{A}{2}$:
	\begin{equation*}
		\sum_{T \in \N} \mu_{s,R}\big(u_0 \hsp : \hsp \sup_{|t| \leq T} \| \Phi_t u_0 \|_{\hsig} > T^{B} \big) < +\infty
	\end{equation*}
	By the Borel-Cantelli lemma, it implies that the set formed by the elements $u_0$ such that $ \sup_{|t| \leq T} \| \Phi_t u_0 \|_{\hsig} > T^{B}$ for infinitely many $T \in \N$ is of zero $\mu_{s,R}$-measure. It means that for $\mu_{s,R}$-almost every $u_0$, there exists $T_0= T_0(u_0)\in \N$ such that:
	\begin{equation*}
		T \geq T_0 \implies \sup_{|t| \leq T} \| \Phi_t u_0 \|_{\hsig} \leq T^{B}
	\end{equation*}
	This implies~\eqref{polynomial growth of Sob norms} for $\mu_{s,R}$-almost every $u_0$, and since $R>0$ is arbitrary, it also holds for $\mu_s$-almost every $u_0$.
\end{proof}

%% file: renormalized_energy_and_cutoff_gm.tex
In this section, we consider $s \in (\frac{5}{4}, \frac{3}{2}]$ and $\sigma < s - \frac{1}{2}$ (close to $s - \frac{1}{2}$). The restriction $s \leq \frac{3}{2}$ indicates that we are interested in Gaussian measures $\mu_s$ supported in Sobolev spaces of regularity (strictly) smaller than 1 (that is below the energy level). The condition $s>\frac{5}{4}$ corresponds to a regime where we are able to make sense of the \textit{renormalized energy} as limit the of truncated renormalized energies~\eqref{renormalized energies}.

\begin{prop}\label{prop cvgce renormalized energies}
	Let $\frac{5}{4} < s \leq \frac{3}{2}$.  Then, for every $q \in [1, \infty)$, the sequence $(\cE_N)_N$ is Cauchy in $L^q(d\mu_s)$, and therefore converges to a limit function $\cE : \hsigt \ra \R^{+}$. \\
	As a consequence, introducing a smooth-cutoff function $\chi : \R \ra [0,1]$ such that $\chi \equiv 1$ on $[-\frac{1}{2},\frac{1}{2}]$ and $\chi \equiv 0$ on $[-1,1]^c$, we obtain that for every $R>0$:
	\begin{equation}\label{cvgce in Lq cutoff}
		\chi_R \circ \cE_N \tendsto{N \ra \infty} \chi_R \circ \cE, \hspace{0.3cm} \textnormal{in $L^q(d\mu_s)$} 
	\end{equation} 
	where $\chi_R = \chi(\cdot/R)$.
\end{prop}

\begin{proof}
	Note first that, thanks to Lemma~\ref{lem Y in W^sig,infty}, $u \mapsto \| u \|_{L^2(\T)}^2$ already belongs to $L^q(d\mu_s)$, and that $u \mapsto \| \pi_N u \|_{L^{p+1}(\T)}^{p+1}$ converges to $u \mapsto \| u \|_{L^{p+1}(\T)}^{p+1}$ in $L^q(d\mu_s)$, $q \in [1,\infty)$. Thus, we are reduced to prove that: 
	\begin{equation*}
		 u \longmapsto \|  \p_x \pi_N u \|_{L^2(\T)}^2 - \sigma_N
	\end{equation*} 
	is a Cauchy sequence in $L^q(d\mu_s)$. Let then $M>N$ integers. From Wiener chaos (see Lemma~\ref{lem Wiener Chaos}):
	\begin{multline}\label{renormalized H1 norm in Lq(mus)}
			\big\| \big( \frac{1}{2} \|  \p_x \pi_M u \|_{L^2(\T)}^2 - \sigma_M\big) -  \big( \frac{1}{2}\|  \p_x \pi_N u \|_{L^2(\T)}^2 - \sigma_N\big) \big\|_{L^q(d\mu_s)} \\ 
			= \frac{1}{2} \E \big[ \big| \sum_{N <|n| \leq M} \frac{|n|^2}{\langle n \rangle^{2s}} (|g_n|^2-1) \big|^q \big]^{\frac{1}{q}} \lesssim q \E \big[ \big| \sum_{N <|n| \leq M} \frac{|n|^2}{\langle n \rangle^{2s}} (|g_n|^2-1) \big|^2 \big]^{\frac{1}{2}} \\ 
			\lesssim  \big( \sum_{N < |n| \leq M} \langle n \rangle^{4(1-s)} \big)^{\frac{1}{2}}
	\end{multline}
This concludes the proof of the first point in Proposition~\ref{prop cvgce renormalized energies}, since the right-hand side tends to 0 as $N,M \ra \infty$ when $s>\frac{5}{4}$. \\

Let us now prove the second point. From the convergence in $L^q(d\mu_s)$, we deduce that $\cE_N$ converges to $\cE$ in $\mu_s$-measure; and, $\chi_R$ being uniformly continuous, we obtain that $\chi_R \circ \cE_N$ converges to $\chi_R \circ \cE$ in $\mu_s$-measure. Besides, $\chi_R \circ \cE_N$ is uniformly bounded in $L^q(d\mu_s)$, $q \in (1,\infty)$. Thus, from the two last points, we deduce that $\chi_R \circ \cE_N$ converges to $\chi_R \circ \cE$ in $L^{q'}(d\mu_s)$ for every $q' \in [1,q)$; and as $q \in (1,\infty)$ is arbitrary, the convergence holds for any $q' \in [1,\infty)$.  
\end{proof}

\begin{prop}\label{prop approx of rho by rhoN} Let $R>0$. For any $\eps>0$, there exists $N_\eps \in \N$ such for every $N \geq N_\eps$, and every Borel set $A \subset \hsig$, 
	\begin{equation*}
		|\mu_{s,R}(A) - \mu_{s,R,N}(A)| \leq \eps
	\end{equation*}
\end{prop}

\begin{proof}
	This follows from the convergence of $\chi_R \circ \cE_N$ to $\chi_R \circ \cE$ in $L^1(d\mu_s)$ (see Proposition~\ref{prop cvgce renormalized energies}).
\end{proof}

\subsection{Estimates with bounded renormalized energy} This paragraph is devoted to a proposition that highlights the benefits of working with bounded renormalized energies. We use the following notation:
\begin{notn}
	For a real number $\gamma$, the symbol $\gamma+$ refers to a real number greater than $\gamma$ that can be chosen arbitrarily close to it.
\end{notn}

\begin{prop}\label{prop estimates with bounded renormalized energy} Let $s \in (\frac{5}{4}, \frac{3}{2}]$.     
	Let $N \in \N$,  $R>0$, and $V \in H^s$. Denote $U:=\pi_N(Y+V)$, where $Y \sim \mu_s$ is the Gaussian random variable defined in~\eqref{Gaussian rva}. Suppose that:
	\begin{equation*}
		\cE_N(U) \leq R 
	\end{equation*}
	 Then, for every $\alpha < s -\frac{1}{2}$, there exists a positive random variable $W_{s,R,N}$ such that:
	\begin{equation}\label{H alpha estimate with bdd renormalized energy}
		\| U \|_{H^\alpha} \leq W_{s,R,N} (1+ \| V \|_{H^s})^{\alpha(3-2s)+}
	\end{equation} 
	which satisfies, for every $m>0$,
	\begin{equation}\label{condition moments of W_s,R,N}
		\sup_{N \in \N} \E[W_{s,R,N}^m] \leq C_{s,R} < \infty
	\end{equation}
	As a consequence,  for every $\eps>0$, there exists a constant $C_{\eps}>0$ such that:
	\begin{equation}\label{nonlinearity estimate with bdd renormalized energy}
		 \|\pi_N(|\pi_N  U|^{p-1}\pi_N  U)\|_{H^{\frac{1}{2} + \eps }(\T)} \leq C_{\eps}  W'_{s,R,N} (1+ \| V \|_{H^s})^{(3-2s)[\frac{p}{2} + \eps]+}
	\end{equation}
	for a positive random variable $W'_{s,R,N}$, with all moment finite (uniformly in $N \in \N$), as in~\eqref{condition moments of W_s,R,N}. 
\end{prop}

\begin{proof} For commodity, we denote $V_N := \pi_N V$ and $Y_N := \pi_N Y$.
	First, by assumption, we have:
	\begin{equation*}
		\cE_N(U) =  \| U \|_{L^2(\T)}^2 + \Big| \frac{1}{p+1} \|  U \|_{L^{p+1}(\T)}^{p+1} + \big( \frac{1}{2} \| \p_x U \|^2_{L^2(\T)} - \sigma_N \big) \Big| \leq R
	\end{equation*}
	So, as an immediate consequence,
	\begin{equation}\label{L2 norm bounded}
		\| U \|_{L^2}^2 \leq R
	\end{equation}
	Next, we expand the term $\frac{1}{2} \| \p_x U \|^2_{L^2} - \sigma_N$ as:
	\begin{equation*}
		\frac{1}{2} \| \p_x U \|^2_{L^2} - \sigma_N = \frac{1}{2} \| \p_x V_N \|^2_{L^2} + \Re \int_\T \p_x V_N \cjg{\p_x Y_N}dx +  \frac{1}{2} \| \p_x Y_N \|^2_{L^2} - \sigma_N
	\end{equation*}
	and we deduce that:
	\begin{multline*}
		\frac{1}{2} \| \p_x V_N\|^2_{L^2} \leq \cE_N(U) +  \big| \int_\T \p_x V_N \cjg{\p_x Y_N}dx \big| + \big| \frac{1}{2} \| \p_x Y_N \|^2_{L^2} - \sigma_N \big| \\
		\leq R +  \| Y_N\|_{H^{s-\frac{1}{2}-}} \|V_N \|_{H^{\frac{5}{2}-s +}} +  \big| \frac{1}{2} \| \p_x Y_N \|^2_{L^2} - \sigma_N \big|
	\end{multline*}
	Using~\eqref{L2 norm bounded}, it follows that
	\begin{equation}\label{first majoration H1 norm of V}
		\frac{1}{2} \| V_N\|_{H^1}^2 \leq 3 R +  \| Y_N\|_{H^{s-\frac{1}{2}-}} \|V_N \|_{H^{\frac{5}{2}-s +}} + 2 \|Y_N\|_{L^2}^2 + \big| \frac{1}{2} \| \p_x Y_N \|^2_{L^2} - \sigma_N \big|
	\end{equation}
	We interpolate now $H^{\frac{5}{2}-s+}$ between $H^1$ and $H^s$ (note that $ 1 \leq \frac{5}{2}-s < \frac{5}{4}<s$) and then apply Young's inequality:
	\begin{multline}\label{interpolation H1 Hs and Young}
		 \| Y_N\|_{H^{s-\frac{1}{2}-}} \|V_N \|_{H^{\frac{5}{2}-s +}} \leq  \| Y_N\|_{H^{s-\frac{1}{2}-}} \|V_N \|^{2\frac{s-\frac{5}{4}}{s-1} - }_{H^1} \|V_N \|_{H^s}^{\frac{\frac{3}{2}-s}{s-1}+} \\
		 \leq   C \| Y_N\|^M_{H^{s-\frac{1}{2}-}} + \frac{1}{4} \|V_N \|^2_{H^1} + \|V_N \|^{2(3-2s)+}_{H^s}
	\end{multline}
	where $C,M>0$ are large constants. Here we used the following identities:
	\begin{equation*}
		\begin{cases}
			\frac{5}{2}-s+ = \big(1- (\frac{\frac{3}{2}-s}{s-1} +)\big) \cdot 1 + \big( \frac{\frac{3}{2}-s}{s-1} + \big) \cdot s \hspace{0.3cm} \textnormal{(for the interpolation)}\\
			1 = ( \frac{s-\frac{5}{4}}{s-1}+) + (1 - \frac{s-\frac{5}{4}}{s-1} -) + \frac{1}{M} \hspace{0.3cm} \textnormal{(for Young's inequality)}
		\end{cases}
	\end{equation*}
	We emphasize that the second identity is valid since $M$ can be chosen arbitrary large. \\
	Finally, plugging~\eqref{interpolation H1 Hs and Young} into~\eqref{first majoration H1 norm of V} yields:
	\begin{equation*}
		\frac{1}{4} \| V_N\|_{H^1}^2 \leq 3 R + \|V_N \|^{2(3-2s)+}_{H^s} + C \| Y_N\|^M_{H^{s-\frac{1}{2}-}} + 2 \|Y_N\|_{L^2}^2 + \big| \frac{1}{2} \| \p_x Y_N \|^2_{L^2} - \sigma_N \big|
	\end{equation*}
	Hence,
	\begin{equation}\label{majoration H1 by Hs}
		\| V_N\|_{H^1} \leq W_{s,R,N}(1 + \| V_N \|_{H^s})^{3-2s+}
	\end{equation}
	where,
	\begin{equation}\label{W_N,R}
		W_{s,R,N} \lesssim \big( R + \| Y_N\|^M_{H^{s-\frac{1}{2}-}} +  \|Y_N\|_{L^2}^2 + \big| \frac{1}{2} \| \p_x Y_N \|^2_{L^2} - \sigma_N \big| \big)^{\frac{1}{2}}
	\end{equation}
	Let us now conclude the proof. Fix $\alpha < s- \frac{1}{2}$. Then by interpolation, by~\eqref{L2 norm bounded}, and by~\eqref{majoration H1 by Hs}, we can write:
	\begin{multline*}
		\| U \|_{H^\alpha} \leq \| Y_N \|_{H^\alpha} + \| V_N \|_{H^\alpha} \leq  \| Y_N \|_{H^\alpha} + \| V_N \|_{L^2}^{1-\alpha} \| V_N \|_{H^1}^\alpha \\ \leq \| Y_N \|_{H^\alpha} + (\| U \|_{L^2} +  \| Y_N \|_{L^2})^{1-\alpha} W_{s,R,N}^\alpha (1 + \| V_N \|_{H^s})^{\alpha(3-2s)+} \\
		\leq W_{s,R,N} (1 + \| V_N \|_{H^s})^{\alpha(3-2s)+}
	\end{multline*}
	where,
	\begin{equation*}
		W_{s,R,N} =  \| Y_N \|_{H^\alpha} + ( \sqrt{R} +  \| Y_N \|_{L^2})^{1-\alpha} W_{s,R,N}^\alpha
	\end{equation*}
	Thanks to Lemma~\ref{lem Y in W^sig,infty},~\eqref{renormalized H1 norm in Lq(mus)}, and the expression of $W_{s,R,N}$ in~\eqref{W_N,R}, we indeed obtain that the $W_{s,R,N}$ satisfy the condition~\eqref{condition moments of W_s,R,N}. \\
	
	To obtain~\eqref{nonlinearity estimate with bdd renormalized energy}, we use the standard estimate:
	\begin{equation*}
		 \|\pi_N(|\pi_N  U|^{p-1}\pi_N  U)\|_{H^{\frac{1}{2} + \eps }(\T)} \leq C_{\eps} \| U \|_{H^{\frac{1}{2} + \eps }(\T)} \| U \|^{p-1}_{H^{\frac{1}{2}}(\T)}
	\end{equation*}
	and then apply~\eqref{H alpha estimate with bdd renormalized energy}. This completes the proof.
\end{proof}

%% file: Poincare_Dulac.tex
This section is organized in two paragraphs; in the first one, we introduce a "correction to the energy" thanks to a Poincaré-normal form reduction, and in the second one we prove a related deterministic estimate. In our analysis, this normal form is essentially an integration by part in time in the difference $ \| \Phi^N_{-t}u\|_{H^s(\T)}^2 - \| \pi_N u\|_{H^s(\T)}^2$ appearing in the formula of the Radon-Nikodym derivatives~\eqref{trsprt of gm}. In particular, we gain an inverse factor of the resonant function $\Omega$ (see Definition~\eqref{def psi and omega}). It has already been performed in~\cite{sun2023quasi} (on the three dimensional torus) and~\cite{knez24transport}, for $p=5$. For a general odd power $p \geq 5$, the computations are identical and are thus not provided here. Also, note that normal form reductions has already been used in the context of quasi-invariance, see for example \cite{oh2017quasi,oh2018optimal,oh_soeng2021quasi,forlano_and_trenberth2019transport}.

\subsection{Poincaré-Dulac normal form reduction}

\begin{defn}\label{def energy functionals}
	Firstly, we define $\cM_s : C^\infty(\T)^{p+1} \ra \R$ and $\cT_s : C^\infty(\T)^{p+1} \ra \R$ the two following multi-linear maps:
	\begin{equation*}
		\cM_s(u_1,...,u_{p+1}):= \sum_{\lincp} \Psi^{(0)}_{2s}(\vec{k}) \prod_{\substack{j=1 \\ odd}}^{p}\widehat{u}_j(k_j) \prod_{\substack{j=2 \\ even}}^{p+1} \cjg{\widehat{u}_j}(k_j) 
	\end{equation*}
	and,
	\begin{equation*}
		\cT_s(u_1,...,u_{p+1}):= \sum_{\lincp} \Psi^{(1)}_{2s}(\vec{k})\prod_{\substack{j=1 \\ odd}}^{p}\widehat{u}_j(k_j) \prod_{\substack{j=2 \\ even}}^{p+1} \cjg{\widehat{u}_j}(k_j) 
	\end{equation*}
	where we recall that $\Psi^{(0)}_{2s}$ and $\Psi^{(1)}_{2s}$ are defined in Definition~\ref{def psi and omega}.
	Then, for $N\in \N$, we define the truncated version of these maps as~:
	\begin{align*}
		\cT_{s,N}(u_1,...,u_{p+1}) &:= \cT_s(\pi_N u_1,...,\pi_N u_{p+1}), & \cM_{s,N}(u_1,...,u_{p+1}) &:= \cM_s(\pi_N u_1,...,\pi_N u_{p+1})
	\end{align*}
	Finally, we consider:
	\begin{equation*}
		\cN_{s,N}(u_1,...,u_{p+1}) := \cT_s(\pi_N(|\pi_N u_1|^{p-1} \pi_N u_1),\pi_Nu_2,...,\pi_Nu_{p+1}) 
	\end{equation*}
\end{defn}

\begin{defn}[Energy correction and modified energy]\label{def modified energy and energy correction}
	For $N \in \N$, we define the \textit{energy correction} $R_{s,N}$ as the function :
	\begin{equation}\label{RsN}
		R_{s,N}(u):= \frac{1}{p+1} \Re \hsp \cT_{s,N}(u)
	\end{equation}
	Moreover, we introduce the \textit{modified energy} $E_{s,N}$ as the function:
	\begin{equation}\label{EsN}
		E_{s,N}(u) := \frac{1}{2} \norm{\pi_N u}_{H^s(\T)}^2 + R_{s,N}(u),
	\end{equation}
	and we also consider the \textit{modified energy derivative} as the function~:
	\begin{equation}\label{Qs,N}
		Q_{s,N}(u):= -\frac{1}{p+1} \Im \hsp \cM_{s,N}(u) + \Im \hsp \cN_{s,N}(u)
	\end{equation}
\end{defn}

\begin{prop}[Poincaré-Dulac normal form reduction]\label{prop Poincaré Dulac} For all $N\in \N$, and every $u :\T \ra \C$, we have:
	\begin{equation}\label{d/dt E = Q}
		\frac{d}{dt}E_{s,N}(\Phi^N_t u)  = Q_{s,N}(\Phi^N_t u)
	\end{equation}
\end{prop}

For a proof of this proposition, we refer to the computations performed in~\cite{sun2023quasi}, Section~2, (see also~\cite{knez24transport} Section~2).

\subsection{A deterministic estimate}

\begin{lem}[Deterministic dyadic estimates]\label{lem determinstic estimate} Let $s >1$. Let $\Psi_{2s} \in \{ \Psi_{2s}^{(0)}, \Psi_{2s}^{(1)}\}$.
	Let $( f^{(1)}_{k_1} )_{k_1 \in \Z}$,...,$( f_{k_{p+1}}^{(p+1)} )_{k_{p+1} \in \Z}$ be sequences of complex numbers that satisfy $f^{(j)}_{k_j}~=~\1_{|k_j| \sim N_j} f^{(j)}_{k_j} $, where $N_1,  ..., N_{p+1}$ are dyadic integers. Then, there exists a constant $C_s > 0$, only depending on $s$, such that:
	\begin{equation}\label{deterministic estimate}
		\sum_{\lincp } | \Psi_{2s}(\vec{k}) | \prod_{j=1}^{p+1} | f_{k_j}^{(j)} | \leq C_s  N_{(1)}^{2(s-1)+} N_{(3)}^{2} (N_{(7)}...N_{(p+1)})^{\frac{1}{2}} \prod_{j=1}^{p+1} \| f^{(j)}_{k_j}\|_{l^2} 
	\end{equation}
	Furthermore, in the regime $N_{(4)} \ll N_{(3)}$, we have the refined estimate:
	\begin{equation}\label{refined deterministic estimate}
		\sum_{\lincp } | \Psi_{2s}(\vec{k}) | \prod_{j=1}^{p+1} | f_{k_j}^{(j)} | \leq C_s  N_{(1)}^{2(s-1)} (N_{(3)}...N_{(6)}N_{(7)}...N_{(p+1)})^{\frac{1}{2}} \prod_{j=1}^{p+1} \| f^{(j)}_{k_j}\|_{l^2} 
	\end{equation}
	Here, $N_{(1)} \geq ... \geq N_{(p+1)}$ refers to a non-increasing rearrangement of $N_1,...,N_{p+1}$. 
\end{lem}
This lemma has already been proven in the case $p=5$ in~\cite{knezevitch2025qualitativequasiinvariancelowregularity}, Lemma 7.5 and Lemma 7.6, see also~\cite{knez24transport} Lemma 9.7. Note that the proof incorporated implicitly dispersive effects through the $L^6$-Strichartz estimate due to Bourgain in~\cite{Bourgain1993}:
\begin{equation*}
	\| g \|_{L^6(\T_t \times \T_x)} \lesssim \| g \|_{H^{0+}(\T)}
\end{equation*}
It would have been natural to try to mimic the proof for a general $p$ (odd), using the Strichartz estimate:
\begin{equation}\label{L p+1 Strichartz estimate}
	\| g \|_{L^{p+1}(\T_t \times \T_x)} \lesssim \| g \|_{H^{\frac{1}{2}- \frac{3}{p+1}+}(\T)}
\end{equation}
due to Bourgain and Demeter in~\cite{bourgain_demeter_l2_decoupling}. However, the loss of derivatives in~\eqref{L p+1 Strichartz estimate} will lead to a worse estimate than the one stated above.
Here, we prove Lemma~\ref{lem determinstic estimate} for $p\geq 7$, it will follow as a consequence of the known result for $p=5$.

\begin{proof}
	Let $p \geq 7$ (odd). For commodity we assume that $N_1 \geq ... \geq N_{p+1}$, that is $N_j = N_{(j)}$. \\
	\textbullet First, we prove~\eqref{deterministic estimate}. We rearrange the sum as:
	\begin{equation*}
		\sum_{\lincp } | \Psi_{2s}(\vec{k}) | \prod_{j=1}^{p+1} | f_{k_j}^{(j)} | = \sum_{k_7,...,k_{p+1}} \prod_{j=7}^{p+1} | f_{k_j}^{(j)} | \sum_{ \substack{k_1-k_2+...-k_6 \\ = -k_7 +...+k_{p+1}}} | \Psi_{2s}(\vec{k}) | \prod_{j=1}^{6} | f_{k_j}^{(j)} |
	\end{equation*}
	and we know that (Lemma 7.5 in~\cite{knezevitch2025qualitativequasiinvariancelowregularity}):
	\begin{equation*}
		\sum_{ \substack{k_1-k_2+...-k_6 \\ = -k_7 +...+k_{p+1}}} | \Psi_{2s}(\vec{k}) | \prod_{j=1}^{6} | f_{k_j}^{(j)} | \leq C_s  N_{(1)}^{2(s-1)+} N_{(3)}^{2} \prod_{j=1}^{6} \| f^{(j)}_{k_j}\|_{l^2} 
	\end{equation*}
	Thus,
	\begin{equation*}
		\sum_{\lincp } | \Psi_{2s}(\vec{k}) | \prod_{j=1}^{p+1} | f_{k_j}^{(j)} | \leq C_s  N_{(1)}^{2(s-1)+} N_{(3)}^{2} \prod_{j=1}^{6} \| f^{(j)}_{k_j}\|_{l^2} \sum_{k_7,...,k_{p+1}} \prod_{j=7}^{p+1} | f_{k_j}^{(j)} | 
	\end{equation*}
	Using now the Cauchy-Schwarz inequality, and the localization of the $f^{(j)}_{k_j}$, leads to~\eqref{deterministic estimate}. \\
	
	\textbullet Second, we prove~\eqref{refined deterministic estimate}. Here, we are in the regime $N_{(4)} \ll N_{(3)}$, so we are in a position to use the lower bound on $\Omega$ from Lemma~\ref{lem lower bound Omega when three high and the rest is low freq}. Besides, thanks to Remark~\ref{rem Omgk = 0 implies N4 sim N3}, we see that the case $\Psi_{2s}(\vec{k})=\Psi^{(0)}_{2s}(\vec{k})$ leads to a zero-contribution, so we only need to treat the case $\Psi_{2s}(\vec{k})=\Psi^{(1)}_{2s}(\vec{k})$. Now, combining~\eqref{psi/Omg estimate} and~\eqref{lower bound Omg three high and the rest is low freq}, we can write:
	\begin{equation*}
		\sum_{\lincp } | \Psi_{2s}(\vec{k}) | \prod_{j=1}^{p+1} | f_{k_j}^{(j)} | \leq C_s N_{(1)}^{2(s-1)} \sum_{k_3,...,k_{p+1}} \prod_{j=3}^{p+1}  | f_{k_j}^{(j)} | \sum_{k_2} |f^{(2)}_{k_2}f^{(1)}_{k_2 - k_3 + ... +k_{p+1}} |  
	\end{equation*}
	Applying the Cauchy-Schwarz inequality first in the $k_2$ summation and then in the $k_3,....,k_{p+1}$ summation leads to~\eqref{refined deterministic estimate}.
\end{proof}

%% file: uniform_Lp_bound_RND.tex
This section is central; it is dedicated to the proof of Proposition~\ref{prop RND in Lq uniformly in N}, which provides a (uniform in $N$) $L^q(d\mu_{s,R,N})$-bound on the truncated Radon-Nikodym derivatives $g_{s,N,t}$. To do so, we combine the Boué-Dupuis variational formula (see Lemma~\ref{lem boue-dupuis}), the Poincaré-Dulac normal form reduction from Proposition~\ref{prop Poincaré Dulac}, introducing \textit{weighted Gaussian measures}, and finally the renormalized-energy cutoff through Proposition~\ref{prop estimates with bounded renormalized energy}.

\subsection{The weighted Gaussian measures}
The idea is to replace the energy $\frac{1}{2} \| . \|^2_{H^s}$ in the formula of the truncated Radon-Nikodym derivatives in~\eqref{trsprt of gm} by the modified energy $E_{s,N}$ in~\eqref{EsN} produced by the Poincaré-Dulac normal form reduction.
\begin{defn}
	For $N \in \N$ and $R>0$, we define the \textit{weighted Gaussian measure}:
	\begin{equation}\label{wgm}
		\rho_{s,R,N} := e ^{-R_{s,N}(u)} \mu_{s,R,N}
	\end{equation}
	where $R_{s,N}$ is the energy correction in~\eqref{RsN}, and $\mu_{s,R,N}$ is the cutoff Gaussian measure in~\eqref{cutoff gm}.
\end{defn}

From the forthcoming Lemma~\ref{lem exp integrability Fs,N}, we have, for every $q \in [1,\infty)$:
\begin{equation}\label{exp RsN in Lq}
	\| e^{- R_{s,N}} \|_{L^q(d\mu_{s,R,N})} \leq C_{s,R,q} < +\infty
\end{equation}
which ensures that $\rho_{s,R,N}$ has a finite mass. As expected, the weighted Gaussian measures are transported by the truncated flow as follows:
\begin{prop}\label{prop trspt wgm}
	Let $N \in \N$, $t \in \R$, and $R>0$. Then, 
	\begin{align}\label{trspt wgm}
		(\Phi^N_t)_\# \rho_{s,R,N} &= f_{s,N,t} d\rho_{s,R,N}, & f_{s,N,t} &:= \exp \big( -(E_{s,N}(\Phi^N_{-t}(u)) - E_{s,N}(u)) \big) 
	\end{align}
	where $\rho_{s,R,N}$ is the weighted Gaussian measure defined in~\eqref{wgm}
\end{prop}

\begin{proof}
	Lemma~\ref{lem trspt density meas} combined with~\eqref{trspt cutoff gm} gives:
	\begin{align*}
		(\Phi^N_t)_\# \rho_{s,R,N} &= (\Phi^N_t)_\# \big( e^{-R_{s,N}(u)} d\mu_{s,R,N}\big) = e^{-R_{s,N}(\Phi^N_{-t}u)} g_{s,N,t} d\mu_{s,R,N} \\
		&= e^{-(R_{s,N}(\Phi^N_{-t}u) - R_{s,N}(u))} g_{s,N,t} d\rho_{s,R,N}
	\end{align*}
	which is the desired equality recalling the definitions~\eqref{trsprt of gm} and~\eqref{EsN}.
\end{proof}

\subsection{Two lemmas} In this paragraph we establish two lemmas that, when combined, imply the desired Proposition~\ref{prop RND in Lq uniformly in N}; for convenience, we devote the next paragraph~\ref{subsection proof of the Lq integrability} to the proof of this implication.\\
The first lemma is inspired by~\cite{forlano2022quasi,coe2024sharp,forlano2025improvedquasiinvarianceresultperiodic}. Here, we show that it is also compatible with weighted Gaussian measures. Again, we will benefit from the conservation of the cutoff; in~\cite{coe2024sharp} the authors were confronted to a situation were this property was not satisfied.

\begin{lem}\label{lem reduction Lq bound RND at t=0} Let $t \in \R$ and $N \in \N$. Then, for every $q \in (1,\infty)$,  
	\begin{equation}\label{estimate with renormalized cutoff energy}
		\sup_{\tau \in [0,t]} \| f_{s,N,\tau}\|^q_{L^q(\rho_{s,R,N})} \leq \int \exp \big( q |t Q_{s,N}(u)| \big) d\rho_{s,R,N}
	\end{equation}
\end{lem}

\begin{rem}
	In~\cite{forlano2022quasi}, the (analogous) term $Q_{s,N}$ could not be considered directly; it was essential to use time oscillations considering instead the term $ \delta^{-1} \int_0^\delta Q_{s,N}(\Phi^N_{t'}u) dt'$, with $\delta > 0$ small. Since this issue does not arise in our situation, we can let $\delta \to 0$, yielding the term $Q_{s,N}(u)$.
\end{rem}

\begin{proof}
		 First, using the definition of the Radon-Nikodym derivative, we have for $\tau \in [0,t]$:
	\begin{equation*}
		\begin{split}
			\| f_{s,N,\tau} & \|^q_{L^q(\rho_{s,R,N})} = \int f_{s,N,\tau}^{q-1} f_{s,N,\tau} d\rho_{s,R,N} \\
			& = \int \exp\big(-(q-1)(E_{s,N}(\Phi^N_{-\tau}u) - E_{s,N}(u)) \big) (\Phi^N_\tau)_\# d\rho_{s,R,N} \\
			& =\int \exp\big((q-1)(E_{s,N}(\Phi^N_{\tau}u) - E_{s,N}(u)) \big)  d\rho_{s,R,N}
		\end{split}
	\end{equation*}
	 Then, we introduce a parameter $n \in \N$, which is intended to tend to infinity. By subdividing $[0,\tau]$ into intervals of length $\frac{\tau}{n}$, applying Jensen's inequality, using again the definition of the Radon-Nikodym derivative, and then Hölder's inequality (considering $q'$ the conjugate exponent of $q$), we obtain:
	\begin{equation*}
		\begin{split}
			\| &f_{s,N,\tau} \|^q_{L^q(\rho_{s,R,N})}  = \int \exp\big( (q-1)n \sum_{k=0}^{n-1} \frac{1}{n}( E_{s,N}(\Phi^N_{(k+1)\frac{\tau}{n}}u) - E_{s,N}(\Phi^N_{k\frac{\tau}{n}}u))\big)  d\rho_{s,R,N} \\
			& \leq \sum_{k=0}^{n-1} \frac{1}{n} \int \exp \big( (q-1)n  ( E_{s,N}(\Phi^N_{\frac{\tau}{n}} \Phi^N_{k\frac{\tau}{n}}u) - E_{s,N}(\Phi^N_{k\frac{\tau}{n}}u)) \big) d\rho_{s,R,N} \\
			& = \sum_{k=0}^{n-1} \frac{1}{n} \int \exp \big( (q-1)n ( E_{s,N}( \Phi^N_{\frac{\tau}{n}}u) - E_{s,N}(u)) \big) f_{s,N,\frac{k \tau}{n}} d\rho_{s,R,N} \\ 
			& \leq \sum_{k=0}^{n-1} \frac{1}{n} \big[ \int \exp \big( q n  ( E_{s,N}( \Phi^N_{\frac{\tau}{n}}u) - E_{s,N}(u)) \big) d\rho_{s,R,N} \big]^{\frac{1}{q'}} \| f_{s,N,\frac{k \tau}{n}} \|_{L^q(d\rho_{s,R,N})} \\
			& \leq \sup_{t' \in [0,t]}\| f_{s,N,t'} \|_{L^q(d\rho_{s,R,N})} \big[ \int \exp \big( q n ( E_{s,N}( \Phi^N_{\frac{\tau}{n}}u) - E_{s,N}(u)) \big) d\rho_{s,R,N} \big]^{\frac{1}{q'}}
		\end{split}
	\end{equation*}
	Thanks to~\eqref{d/dt E = Q}, we obtain by passing to the limit $n \to \infty$:
	\begin{equation}\label{rate of change}
		\| f_{s,N,\tau} \|^q_{L^q(\rho_{s,R,N})}  \leq  \sup_{t' \in [0,t]}\| f_{s,N,t'} \|_{L^q(d\rho_{s,R,N})} \big[ \int \exp \big( q \tau  Q_{s,N}(u) \big) d\rho_{s,R,N} \big]^{\frac{1}{q'}}
	\end{equation}
	and then, passing to the $\sup$ over $\tau \in [0,t]$ yields:
	\begin{equation*}
		\sup_{\tau \in [0,t]}	\| f_{s,N,\tau} \|^q_{L^q(\rho_{s,R,N})} \leq \sup_{\tau \in [0,t]} \int \exp \big( q \tau  Q_{s,N}(u) \big) d\rho_{s,R,N} \leq \int \exp \big( q |t Q_{s,N}(u)| \big) d\rho_{s,R,N}
	\end{equation*}
	which is the desired inequality. To conclude, let us justify the passage to the limit in~\eqref{rate of change}. Here, we can apply the dominated convergence theorem; indeed:
		\begin{equation*}
		n | E_{s,N}( \Phi^N_{\frac{\tau}{n}}u) - E_{s,N}(u)| \leq | \tau | \sup_{t' \in [0,\frac{\tau}{n}]} |\frac{d}{dt'} E_{s,N}(\Phi^N_{t'}u)| =  | \tau | \sup_{t' \in [0,\frac{\tau}{n}]} |Q_{s,N}(\Phi^N_{t'}u)| 
	\end{equation*}
	and thanks to the $L^2(\T)$-norm conservation, the fact that the $L^2(\T)$-norm is bounded in the integral due to the cutoff $\chi_R \circ \cE_N$, and the expression of $Q_{s,N}$ in~\eqref{Qs,N}, we can then write:
	\begin{equation*}
		n | E_{s,N}( \Phi^N_{\frac{\tau}{n}}u) - E_{s,N}(u)| \leq |\tau |C_{R,N} < + \infty
	\end{equation*}
	where $C_{R,N}>0$ is independent of $n$. 
\end{proof}

\begin{lem}\label{lem exp integrability Fs,N}
	Let $s>s_p$, where $s_p$ is defined in~\eqref{sp threshold}. Then, for $A>0$ large enough, we have that for every $R>0$, there exists a constant $C_{s,R}>0$, such that for every $\lambda > 0$, $N\in\N$, and $\cF_{s,N} \in \{ \cM_{s,N}, \cT_{s,N}, \cN_{s,N}\}$,
	\begin{equation}\label{exp integrability Fs,N}
		\log \int  e^{\lambda | \cF_{s,N}(u)|}d\mu_{s,R,N}   \leq C_{s,R} \hsp \lambda^A 
	\end{equation}
	As a consequence, we have:
	\begin{equation}\label{exp integrability RsN and QsN}
		\log \int e^{\lambda | R_{s,N}(u)|}d\mu_{s,R,N}  +  \log \int  e^{\lambda | Q_{s,N}(u)|}d\mu_{s,R,N}   \leq C_{s,R} \hsp \lambda^A 
	\end{equation}
\end{lem}

\begin{notns}[Preparation for the proof of Lemma~\ref{lem exp integrability Fs,N}]\label{notn energy estimate}  Let $N\in \N$. Firstly, for $V \in H^s(\T)$, we denote: \begin{center}
		$U := \pi_N (V +  Y)$, where $Y \sim \mu_s$ is the Gaussian random variable defined in~\eqref{Gaussian rva}.
	\end{center} Secondly, for $\cF_{s,N} \in \{ \cM_{s,N}, \cT_{s,N}, \cN_{s,N} \}$, and considering $F_N : U \mapsto \pi_N(|\pi_N  U|^{p-1}\pi_N  U)$, we observe that there exists a unique
	\begin{equation*}
		(\Psi_{2s},G) \in \{ (\Psi^{(0)}_{2s},id) \}  \cup \{\Psi^{(1)}_{2s} \} \times \{ id, F_N\}
	\end{equation*}	such that:
	\begin{equation*}
		\cF_{s,N}(U)= \sum_{\lincp} \Psi_{2s}(\vec{k}) G_{k_1}\cjg{U_{k_2}}...\cjg{U_{k_{p+1}}}
	\end{equation*}
	with the abuse of notation $G = G(U)$, that we use for commodity (and where $G_k, U_k$ correspond to the $k$-th Fourier coefficient). This concise formulation for $\cF_{s,N}$ will allow us to avoid treating the terms $\cM_{s,N}, \cT_{s,N}$, and $\cN_{s,N}$ separately.\\
	Thirdly, for a set of dyadic integers $(N_1,...,N_{p+1}) \in (2^{\N})^{p+1}$, we denote $\vec{N} = (N_1,...,N_{p+1})$, and we consider the dyadic block for $\cF_{s,N}$:
	\begin{equation}\label{Fs,N dyadic block}
		\cF_{s,\vec{N}}(U) := \sum_{\lincp} \Psi_{2s}(\vec{k}) G^{N_1}_{k_1}\cjg{U^{N_2}_{k_2}}...\cjg{U^{N_{p+1}}_{k_{p+1}}}
	\end{equation} 
	where, for $W: \T \to \C$ and $L \in 2^{\N}$, we adopt the notation $W^L = P_{L} \pi_N W$, with $P_L$ the projector onto frequencies of size $\sim$ L.   \\
	Finally, we denote $N_{(1)} \geq ... \geq N_{(p+1)}$ a non-increasing rearrangement of the dyadic integers $N_1,...,N_{p+1}$. 
\end{notns}

\begin{proof}[Proof of Lemma~\ref{lem exp integrability Fs,N}]
	Let $N \in \N$ and $\cF_{s,N} \in \{ \cM_{s,N}, \cT_{s,N}, \cN_{s,N} \}$. We adopt the Notations~\ref{notn energy estimate}. We start by applying the simplified Boué-Dupuis formula~\eqref{ineq boue dupuis} to the function $\cF = \lambda \1_{\cE_N \leq R}|\cF_{s,N}|$:
	\begin{align*}
		\log \int  e^{\lambda | \cF_{s,N}(u)|}d\mu_{s,R,N} &\leq \log \int  e^{\1_{\cE_N(u) \leq R} \lambda | \cF_{s,N}(u)|}  d\mu_s \\ 
	&	\leq \E \big[\sup_{V \in H^s} \{ \lambda \1_{\cE_N(U) \leq R}|\cF_{s,N}(U)| - \frac{1}{2}\norm{V}^2_{H^s} \} \big]
	\end{align*}
	Then, decomposing $\cF_{s,N}$ into dyadic blocks, we obtain:
	\begin{equation}\label{decomposition into dyadic block in boue-dupuis}
		\log \int  e^{\lambda | \cF_{s,N}(u)|}d\mu_{s,R,N} \leq \E \big[\sup_{V \in H^s} \{ \lambda \1_{\cE_N(U) \leq R}\sum_{\vec{N}} | \cF_{s,\vec{N}}(U)| - \frac{1}{2}\norm{V}^2_{H^s} \} \big]
	\end{equation}
	Now, our goal is to prove that:
	\begin{equation}\label{energy dyadic block estimate}
		\1_{\cE_N(U) \leq R}|\cF_{s,\vec{N}}(U)| \leq N_{(1)}^{0-} X_{s,R,N} (1 + \| V \|_{H^s})^{2-}
	\end{equation}
	with $X_{s,R,N}$ a non-negative random variable satisfying:
	\begin{equation}\label{all moments finite X_{s,R,N}}
	\forall m > 0,	\hspace{0.3cm} \E [ X_{s,R,N}^m] \leq C_{s,R,m} < +\infty 
	\end{equation}
	Indeed, if we do so, then we will be able to continue~\eqref{decomposition into dyadic block in boue-dupuis} as follows: 
	\begin{equation*}
		\log \int  e^{\lambda | \cF_{s,N}(u)|}d\mu_{s,R,N} \lesssim \E\big[ \sup_{V \in H^s} \lambda X_{s,R,N}(1 + \| V \|_{H^s})^{2-} -  \frac{1}{2}\norm{V}^2_{H^s} \big] \lesssim \lambda^A \E[X_{s,R,N}^A]
	\end{equation*}
	for $A>0$ large; and finally using the property~\eqref{all moments finite X_{s,R,N}} in the estimate above will ensure~\eqref{exp integrability Fs,N}. \\
	
	Hence, let us turn to the proof of~\eqref{energy dyadic block estimate}. 	In $\cF_{s,\vec{N}}(U)$ (see ~\eqref{Fs,N dyadic block}), only $N_1$ has a special position, due to the presence of $G$. In our analysis, the size of $N_1$ with respect to the other frequencies does not play a crucial role, and in all cases we estimate $G^{N_1}$ applying Proposition~\ref{prop estimates with bounded renormalized energy}. Hence, for convenience, we assume in what follows that $N_j = N_{(j)}$ (for $j=1,...,p+1$), as the other cases are similar. Finally, throughout the proof we use the following fact:
	\begin{equation*}
		\big( \lincp \hspace{0.2cm} \textnormal{and} \hspace{0.2cm} \forall j, \hsp |k_j| \sim N_j \big) \implies N_{(1)} \sim N_{(2)}, \hspace{0.2cm} \textnormal{that is} \hsp N_1 \sim N_2
	\end{equation*}
	without explicitly referring to it.\\
	
	Now, we consider the two following frequency regime:
	\begin{equation*}
		\cF_{s,\vec{N}}(U) = \1_{N_4 \ll N_3} \cF_{s,\vec{N}}(U) + \1_{N_4 \sim N_3} \cF_{s,\vec{N}}(U)
	\end{equation*}
	and we treat separately each of these terms. \\
	
	\textbf{The first regime:} When $N_4 \ll N_3$. We are here in a position to use the refined estimate~\eqref{refined deterministic estimate}. Doing so leads to: 
	\begin{equation*}
		|\cF_{s,\vec{N}}(U)| \lesssim N_{(1)}^{2(s-1)} \| G^{N_1} \|_{L^2} \| U^{N_2} \|_{L^2} \prod_{j=3}^{p+1} \| U^{N_j} \|_{H^{\frac{1}{2}}}
	\end{equation*}
		Applying now \eqref{nonlinearity estimate with bdd renormalized energy} for $G^{N_1}$ (true for $G=F_N$, and also clearly for $G=id$),~\eqref{H alpha estimate with bdd renormalized energy} with $\alpha = 2(s-\frac{5}{4})$ to $U^{N_2}$, and~\eqref{H alpha estimate with bdd renormalized energy} with $\alpha = \frac{1}{2}$ to $U^{N_j}$ ($j \geq 3$), yields:
	\begin{equation*}
		\begin{split}
			\1_{\cE_N(U) \leq R} |\cF_{s,\vec{N}}(U)| &\lesssim N_1^{-\eps} X_{s,R,N} (1+ \| V \|_{H^s})^{(3-2s) [\frac{p}{2} +\eps + 2(s-\frac{5}{4}) + \frac{p-1}{2}] + } \\
			& = N_1^{-\eps} X_{s,R,N} (1+ \| V \|_{H^s})^{(3-2s) [2s +p -3 +\eps] + }
		\end{split}
	\end{equation*}
	where we gathered in $X_{s,R,N}$ all the positive random variables appearing when using Proposition~\ref{prop estimates with bounded renormalized energy}. Thanks to this proposition, we indeed see that $X_{s,R,N}$ satisfies~\eqref{all moments finite X_{s,R,N}}; and then, thanks to Lemma~\ref{lem threshold quasiinv}, the inequality above implies the desired estimate~\eqref{energy dyadic block estimate} for $\eps>0$ small enough. \\
	
	\textbf{The second regime:} When $N_4 \sim N_3$. Here, we apply~\eqref{deterministic estimate} and obtain:
	\begin{equation*}
			\1_{\cE_N(U) \leq R} |\cF_{s,\vec{N}}(U)| \lesssim 	\1_{\cE_N(U) \leq R} \hsp N_1^{2(s-1) +}  N_3^2 \| G^{N_1} \|_{L^2}  \prod_{j=2}^{6} \| U^{N_j} \|_{L^2} \prod_{j=7}^{p+1} \| U^{N_j} \|_{H^{\frac{1}{2}}} 	
	\end{equation*}
	In this inequality, we simply bound $\| U^{N_5} \|_{L^2}$ and $\| U^{N_6} \|_{L^2}$ by $R$ (see~\eqref{L2 norm bounded}). Then, we use again~\eqref{nonlinearity estimate with bdd renormalized energy} for $G^{N_1}$; and, we use~\eqref{H alpha estimate with bdd renormalized energy} with $\alpha= \sigma = s -\frac{1}{2}-$ for $U^{N_2}$ and $U^{N_3}$, and with $\alpha = \frac{1}{2}$ for $U^{N_j}$ ($j=4$ and $j \geq 7$). In doing so we obtain:
	\begin{equation*}
		\begin{split}
			\1_{\cE_N(U) \leq R} |\cF_{s,\vec{N}}(U)| & \lesssim N_1^{s - 2 -\eps + } N_3^{2 - s +} X_{s,R,N} (1+ \| V \|_{H^s})^{(3-2s) [ \frac{p}{2} + \eps + 2(s - \frac{1}{2}) + \frac{p-4}{2}] + }   \\
			& \lesssim N_1^{-\eps +} X_{s,R,N} (1+ \| V \|_{H^s})^{(3-2s) [2s + p - 3 +\eps] + }
		\end{split}     
	\end{equation*}
	with again $X_{s,R,N}$ satisfying~\eqref{all moments finite X_{s,R,N}} (thanks to~\eqref{condition moments of W_s,R,N} in Proposition~\ref{prop estimates with bounded renormalized energy}). Since we can replace in the inequality above $-\eps +$ by $-\eps + \delta_0$, where $\eps$ and $\delta_0$ can be chosen independently, we indeed obtain (thanks to Lemma~\ref{lem threshold quasiinv}) the desired estimate~\eqref{energy dyadic block estimate} (taking $\delta_0 \leq \frac{\eps}{2}$ and $\eps>0$ small enough). The proof is now complete.
\end{proof}

\subsection{Proof of the Lq integrability}\label{subsection proof of the Lq integrability} Here, we prove Prpoposition~\ref{prop RND in Lq uniformly in N} combining the two previous lemmas~\ref{lem reduction Lq bound RND at t=0} and~\ref{lem exp integrability Fs,N}. First, they imply that:
\begin{equation}\label{fs,N,t Lq estimate}
	 \| f_{s,N,t}\|^q_{L^q(\rho_{s,R,N})} \leq  \exp \big( C_{s,R,q} |t|^A \big)
\end{equation}
for every $ t \in \R$, $R >0$ and $q \in (1,\infty)$. Next, recalling the definitions of the Radon-Nikodym derivatives $g_{s,N,t}$ and $f_{s,N,t}$ in~\eqref{trsprt of gm} and~\eqref{trspt wgm} respectively, we write:
	\begin{align*}
	&\| g_{s,N,t} \|^q_{L^q(d\mu_{s,R,N})} = \int   e^{qR_{s,N}(\Phi^N_{-t}u) - qR_{s,N}(u)} f_{s,N,t}^q d\mu_{s,R,N}  \\
	&= \int [e^{qR_{s,N}(\Phi^N_{-t}u)} f_{s,N,t}^{\frac{1}{3}} e^{-\frac{1}{3} R_{s,N}(u)}] [e^{-(q-\frac{2}{3})R_{s,N}(u)}] [f_{s,N,t}^{q-\frac{1}{3}} e^{-\frac{1}{3}R_{s,N}(u)}] d\mu_{s,R,N} \\
	& \leq \big( \int e^{3qR_{s,N}(\Phi^N_{-t}u)} f_{s,N,t} d\rho_{s,R,N}  \big)^{\frac{1}{3}} \big( \int e^{(-3q + 2)R_{s,N}(u)}d\mu_{s,R,N} \big)^{\frac{1}{3}} \big( \int f_{s,N,t}^{3q-1} d\rho_{s,R,N} \big)^{\frac{1}{3}} \\
	& =  \big( \int e^{3qR_{s,N}(u)} d\rho_{s,R,N}  \big)^{\frac{1}{3}} \big( \int e^{(-3q + 2)R_{s,N}(u)}d\mu_{s,R,N} \big)^{\frac{1}{3}} \big( \int f_{s,N,t}^{3q-1} d\rho_{s,R,N} \big)^{\frac{1}{3}} 
\end{align*}
Then, by~\eqref{exp integrability Fs,N} and~\eqref{fs,N,t Lq estimate}, we deduce that:
\begin{equation*}
	\| g_{s,N,t} \|^q_{L^q(d\mu_{s,R,N})} \leq \exp \big( C_{s,R,q}|t|^A \big)
\end{equation*}
as promised.